\documentclass{amsart}
\usepackage[utf8]{inputenc}
\usepackage{amsfonts,amsmath,amssymb,amsthm,color,relsize}
\usepackage[abbrev,alphabetic]{amsrefs}
\usepackage[pdfborder={0 0 0 [0 0 ]},bookmarksdepth=3, bookmarksopen=true,unicode=true]{hyperref}
\usepackage[capitalize]{cleveref}
\usepackage{kbordermatrix}
\usepackage{tikz-cd}
\usepackage{mathrsfs}
\usepackage{enumitem}

\newcommand{\N}{\mathbb{N}}
\newcommand{\Z}{\mathbb{Z}}
\newcommand{\R}{\mathbb{R}}
\newcommand{\Q}{\mathbb{Q}}
\newcommand{\D}{\mathcal{D}}

\newcommand{\ts}{\widetilde{\sigma}}

\newcommand{\p}{\mathcal{P}}

\newcommand{\B}{\mathcal{B}}
\newcommand{\spanY}{\mathrm{span}_{D(Y)[\bar t^{\pm}]}}
\newcommand{\spanX}{\mathrm{span}_{D(X)}}
\DeclareMathOperator{\Span}{span}
\newcommand{\A}{\mathcal{A}}
\renewcommand{\C}{\mathbb{C}}

\renewcommand{\kbldelim}{(}
\renewcommand{\kbrdelim}{)}
\renewcommand{\subset}{\subseteq}

\DeclareMathOperator{\supp}{supp}
\DeclareMathOperator{\cone}{cone}

\DeclareMathOperator{\ord}{ord}
\DeclareMathOperator{\im}{im}

\DeclareMathOperator{\ore}{Ore}

\DeclareMathOperator{\rank}{rk}
\DeclareMathOperator{\GL}{GL}

\DeclareMathOperator{\Det}{\det\nolimits}
\DeclareMathOperator{\tree}{Tree}
\DeclareMathOperator{\rat}{Rat}

\theoremstyle{plain}
\newtheorem{theorem}{Theorem}[section]
\crefname{theorem}{Theorem}{Theorems}

\newtheorem{question}[theorem]{Question}
\newtheorem{proposition}[theorem]{Proposition}
\newtheorem{lemma}[theorem]{Lemma}
\newtheorem{corollary}[theorem]{Corollary}

\theoremstyle{remark}
\newtheorem{remark}[theorem]{Remark}
\theoremstyle{definition}
\newtheorem{definition}[theorem]{Definition}
\newtheorem{example}[theorem]{Example}
\newtheorem{notation}[theorem]{Notation}

\makeatletter
\providecommand*{\twoheadrightarrowfill@}{%
  \arrowfill@\relbar\relbar\twoheadrightarrow
}
\providecommand*{\twoheadleftarrowfill@}{%
  \arrowfill@\twoheadleftarrow\relbar\relbar
}
\providecommand*{\xtwoheadrightarrow}[2][]{%
  \ext@arrow 0579\twoheadrightarrowfill@{#1}{#2}%
}
\providecommand*{\xtwoheadleftarrow}[2][]{%
  \ext@arrow 5097\twoheadleftarrowfill@{#1}{#2}%
}
\makeatother

\title[Agrarian Betti numbers, with a twist]{Agrarian and $\ell^2$-Betti numbers of locally indicable groups, with a twist}
\date{}

\author{Dawid Kielak}
\address{Dawid Kielak, Mathematical Institute, University of Oxford, Andrew Wiles Building,
Radcliffe Observatory Quarter,
Woodstock Road,
Oxford
OX2 6GG,
United Kingdom
}
\email{kielak@maths.ox.ac.uk}
\author{Bin Sun}
\address{Bin Sun, Mathematical Institute, University of Oxford, Andrew Wiles Building,
Radcliffe Observatory Quarter,
Woodstock Road,
Oxford
OX2 6GG,
United Kingdom
}
\email{bin.sun@maths.ox.ac.uk}

\newcounter{dawidcomments}

\begin{document}

\maketitle

\begin{abstract}
 We prove that twisted $\ell^2$-Betti numbers of locally indicable groups are equal to the usual $\ell^2$-Betti numbers rescaled by the dimension of the twisting representation; this answers a question of L\"uck for this class of groups. It also leads to two formulae: given a fibration $E$ with base space $B$ having locally indicable fundamental group, and with a simply-connected fiber $F$, the first formula bounds $\ell^2$-Betti numbers $b_i^{(2)}(E)$ of $E$ in terms of $\ell^2$-Betti numbers of $B$ and usual Betti numbers of $F$; the second formula computes $b_i^{(2)}(E)$ exactly in terms of the same data, provided that $F$ is a high-dimensional sphere.
 
 We also present an inequality between twisted Alexander and Thurston norms for free-by-cyclic groups and $3$-manifolds.
 The technical tools we use come from the theory of generalised agrarian invariants, whose study we initiate in this paper.
\end{abstract}

\section{Introduction}
Motivated by the Gauss--Bonnet theorem, and seeking  \emph{``Zussamenh\"ange und Bindungen [...] zwischen den topologischen Eigenschaften einerseits und den differentialgeometrischen Eigenschaften andererseits''}, Heinz Hopf formulated in 1932 \cite{hopf1932differential}*{Page 224} a somewhat vague question about the relationship between curvature of even-dimensional Riemannian manifolds and their Euler characteristic. The question was then given an explicit form, now known as the Hopf Conjecture. The version of the conjecture for negatively curved manifolds can be found in the work of Yau
\cite{Yau1982}*{Problem 10}), and states: every closed Riemannian $2n$-manifold
$M$ of negative sectional curvature satisfies $(-1)^n\chi(M)>0$, where $\chi(M)$ denotes the Euler characteristic of $M$. 

Since Riemannian manifolds of negative sectional curvature are aspherical, one can replace the geometric assumption by a topological one. This was done by Thurston (see \cite{Kirby1997}*{Problem 4.10}), who formulated the following conjecture: every closed aspherical $2n$-dimensional manifold $M$ satisfies  $(-1)^n\chi(M)\geqslant 0$.

An attack strategy for resolving Hopf Conjecture was proposed by Singer, as reported by Dodziuk \cite{Dodziuk1979}*{Conjecture 2} and Yau~\cite{Yau1982}. What is now known as the Singer Conjecture states that the $\ell^2$-Betti numbers $b_i^{(2)}(M)$ of a closed aspherical $n$-manifold $M$ should vanish in all dimensions, except perhaps the middle dimension $\frac n 2$, if it exists; moreover, if the manifold is negatively curved, then the middle $\ell^2$-Betti number should be strictly positive. Knowing the $\ell^2$-Betti numbers allows one to compute the $\ell^2$-Euler characteristic, which for manifolds is equal to the usual Euler characteristic by the $L^2$ Index Theorem of Atiyah~\cite{Atiyah1976}. Therefore, Singer Conjecture implies Hopf and Thurston Conjectures. Furthermore, Singer Conjecture was established for locally symmetric spaces by Borel~\cite{Borel1985} (see also \cite{Olbrich2002}), a large class of manifolds of classical interest.

Singer Conjecture remains open, at least partly because computing $\ell^2$-Betti numbers of manifolds is notoriously difficult. In this article we give a method of computing $\ell^2$-Betti numbers for  fibrations 
 $F\rightarrow E \rightarrow B$  when $F$ is sufficiently similar to a high-dimensional sphere, and when $\pi_1(B)$ is virtually locally indicable:

\begin{theorem}\label{thm. l2 betti fibration simple}
Let $F\rightarrow E \rightarrow B$ be a fibration of connected finite CW-complexes, with $\pi_1(B)$ being virtually locally indicable. If $F$ is simply connected, or more generally, if the map $\pi_1(E)\rightarrow\pi_1(B)$ induced by the fibration is an isomorphism, then
\[b^{(2)}_i(E)\leqslant \sum^i_{j=0}b_j(F)\cdot b^{(2)}_{i-j}(B)\]
for every $i \in \mathbb N$.

If moreover the homology of $F$ with $\C$-coefficients is non-zero in at most two degrees, $0$ and $n$ with $n\geqslant \max \{2, \dim B \}$, then for every $i \in \mathbb N$ we have
    \[b^{(2)}_i(E)=b^{(2)}_i(B)+b_n(F)\cdot b^{(2)}_{i-n}(B).\]
\end{theorem}

In fact, we prove a more general result, \cref{thm. l2 betti fibration}, whose statement is perhaps too involved for this introduction; \cref{thm. l2 betti fibration simple} summarises items \ref{item. l2 betti 1} and \ref{item. l2 betti 2} of \cref{thm. l2 betti fibration}.

Here, we recall that a group $G$ is \textit{locally indicable} if every non-trivial finitely generated subgroup admits an epimorphism onto $\Z$. A group $G$ is \textit{virtually locally indicable} if $G$ has a finite-index locally indicable subgroup.

An important class of examples satisfying the assumption of the second part of \cref{thm. l2 betti fibration simple} is formed by sphere bundles. Even though such bundles may not be of central importance in group theory, they do lead to interesting examples in topology.
The \textit{Milnor Exotic Sphere} \cites{milnor1956manifolds} is a fiber bundle $S^3\rightarrow M^7\rightarrow S^4$. By computing an invariant of the smooth structure of this fibration, Milnor proved that $M^7$ is homeomorphic, but not diffeomorphic, to $S^7$. There is also the Hopf bundle \cites{hopf1964abbildungen} $S^1\rightarrow S^3\rightarrow S^2$ which shows that $\pi_3(S^2)\neq 0$. Heuristically, by replacing the base with a much more general manifold, one should be able to construct many more exotic manifolds.

\smallskip
Using techniques similar to those underpinning \cref{thm. l2 betti fibration simple}, in \cref{sec. twisted l2 betti} we will prove the following results.

\begin{corollary}\label{cor. surface base}
    Let $F\rightarrow E \rightarrow B$ be a fiber bundle of compact connected manifolds such that $F$ is simply connected and $B$ is a surface with $|\pi_1(B)|=\infty$ (that is, except when $B$ is either $S^2$, the $2$-disk, or the projective plane $P^2$). Then for every $i \in \mathbb N$ we have
    \[b^{(2)}_i(E)=-\chi(B)b_{i-1}(F).\]
\end{corollary}

\begin{corollary}\label{cor. 3-mnfl base}
    Let $F\rightarrow E \rightarrow B$ be a fiber bundle of compact connected manifolds such that $F$ is simply connected and $B$ is an orientable irreducible $3$-manifold. Suppose further that $B$ either has positive virtual first Betti number or is geometric or is non-positively curved. Then $b^{(2)}_{\ast}(E)=0$.
\end{corollary}

In the setting of all of the above results, the $\ell^2$-homology of $E$ is related via the Leray--Serre spectral sequence to the homology of $B$ with $\ell^2$-coefficients twisted by the action of $\pi_1(B)$ on the usual $\C$-homology of $F$, as noticed by L\"{u}ck \cite{lueck2018twisting} (see \cref{sec. twisted l2 betti}). Our proof of Theorem \ref{thm. l2 betti fibration} thus follows from relating the \textit{twisted $\ell^2$-Betti number} of $B$ to $b^{(2)}_{\ast}(B)$.

In fact, the question of the nature of twisted $\ell^2$-Betti numbers can be asked more generally, in the realm of group theory; in particular, for a group $G$ and an $n$-dimensional complex representation $\sigma$, L\"uck asked whether the $\ell^2$-Betti numbers of $G$ twisted by $\sigma$ are equal to the corresponding usual $\ell^2$-Betti numbers  multiplied by $n$ (see Question \ref{q. luck}). In \cref{thm. twisted l2 betti} we prove that this is indeed the case when $G$ is locally indicable.

\smallskip

Another setting in which one wishes to twist homology with finite-dimensional representations is that of knots and $3$-manifolds. Here one studies for example the twisted Alexander polynomial, which was introduced, as a refined version of the classical Alexander polynomial, by Lin \cites{lin2001representations} for knot groups and Wada \cites{wada1994twisted} for general finitely presented groups. It is used to distinguish certain knots from their inversions \cites{kirk1999twisted}. Twisted Alexander polynomials can also detect fiberedness of characters, as proven by Friedl--Vidussi~\cite{FriedlVidussi2008}. This fact is crucially used in two important recent results of Jaikin-Zapirain~\cite{Jaikin2020} and Liu~\cite{Liu2020}.
We refer the reader to \cites{friedl2011survey} for a survey on this topic.

Every twisted Alexander polynomial leads to a \emph{twisted Alexander norm}, a function  $H^1(G,\R) \to [0,\infty)$. The $\ell^2$-analogue of this for locally indicable groups is the Thurston norm  $H^1(G,\R) \to [0,\infty)$, which is in fact a semi-norm. In the untwisted setting, for free-by-cyclic groups and many $3$-manifold groups, the Alexander norm is also a semi-norm with values bounded above by the Thurston norm \cites{mcmullen2002alexander,funke2018alexander}. We show that the same is true for twisted Alexander norms: 

\begin{theorem}[Theorems \ref{thm. free-by-cyclic} and \ref{thm. 3-mnfl}]
Let $G$ be either
\begin{itemize}
    \item a (finitely generated free)-by-(infinite cyclic) group, or
    \item the fundamental group of a closed connected orientable $3$-manifold that fibers over $S^1$.
\end{itemize}
Then for every finite-dimensional representation $\sigma\colon G\rightarrow \GL_n(\C)$ and every character $\phi\in H^1(G,\Z)$ we have
\[\|\phi\|_{\sigma}\leqslant n\cdot\|\phi\|_T,\]
where $\|\phi\|_{\sigma}$ (resp. $\|\phi\|_T$) denotes the twisted Alexander norm of $\phi$ with respect to $\sigma$ (resp. the Thurston norm of $\phi$). Moreover, equality holds when $\phi$ is a \textit{fibered character}, i.e., when $\ker\phi$ is finitely generated.
\end{theorem}
The $3$-manifold case of the above theorem was first proved by Friedl--Kim \cites{friedl2008twisted} by a different method.

Both the Alexander polynomial and $\ell^2$-homology form part of the unified theory of agrarian invariants; the same holds true for their twisted analogues -- both twisted $\ell^2$-homology and twisted Alexander polynomials are manifestations of \emph{agrarian invariants} of a generalized kind, whose study we initiate here.

\smallskip

\noindent\textbf{Outline of the paper.} We recall the necessary definitions and results in Section \ref{sec. prelim}. We then introduce agrarian invariants (of a generalized kind) in Section \ref{sec. agrarian notions}. We study twisted $\ell^2$-Betti numbers in Section \ref{sec. twisted l2 betti}, where Theorem \ref{thm. l2 betti fibration simple} and Corollaries \ref{cor. surface base}, \ref{cor. 3-mnfl base} are proved. The rest of the paper, i.e., Sections \ref{sec. euler characteristic} through \ref{sec. application}, are devoted to the study of the Thurston and twisted Alexander norms.

\bigskip

\noindent\textbf{Acknowledgements.} We are very grateful to Andrei Jaikin-Zapirain for pointing out a gap in an early version of this paper.

This work has received funding from the European Research Council (ERC) under the European Union’s Horizon 2020 research and innovation programme (Grant agreement No. 850930).

\section{Preliminaries}\label{sec. prelim}

\subsection{Twisted group rings}
We now recall the construction of twisted group rings, which will be used in dealing with twisted $\ell^2$-Betti numbers.

\begin{definition}
Let $R$ be an associative ring with unity. Denote the group of units of $R$ by $R^{\times}$. Let $G$ be a group and let 
\begin{align*}
    c&\colon G\rightarrow \mathrm{Aut}(R), g\mapsto c_g,\\
    \tau&\colon G\times G\rightarrow R^{\times}, (g,g')\mapsto \tau(g,g')
\end{align*}
be functions such that
\[c_g(c_{g'}(r))=\tau(g,g')\cdot c_{gg'}(r)\cdot \tau(g,g')^{-1},\]
\[\tau(g,g')\tau(gg',g'')=c_g(\tau(g',g''))\cdot \tau(g,g'g''),\]
where $g,g',g''\in G$ and $r\in R$. The pair $(c,\tau)$ is the pair of \textit{structure functions}. We denote by $RG$ the free $R$-module with basis $G$ and write elements of $RG$ as finite $R$-linear combinations $\sum_{g\in G}r_g\ast g$ of elements of $g$. When convenient, we shorten $1\ast g$ to $g$. The structure functions endow $RG$ with the structure of an (associative) \textit{twisted group ring} by declaring
\[(r\ast g)\cdot (r'\ast g')=(r\cdot c_g(r')\cdot \tau(g,g'))\ast (gg')\]
and extending linearly.

If $c_g=\mathrm{id}_R$ and $\tau(g,g')=1$ for all $g,g'\in G$, in which case we say the structure functions $(c,\tau)$ are \textit{trivial}, then $R G$ is called the \textit{untwisted group ring}, in which case we will write elements of $RG$ as $\sum_{g\in G} rg$ instead of $\sum_{g\in G}r\ast g$, and we will also shorten $r\ast 1$ to $r$.
\end{definition}

Note that $\tau(1,1)^{-1}\ast 1$ is the multiplicative unit of the twisted group ring.

\begin{notation}
We will mostly use the notation $RG$ for a twisted group ring. However, in Section \ref{sec. rationalization} we will talk about two group ring structures on $RG$, one twisted and the other one untwisted. There, we will denote the twisted group ring by $R\ast G$ and the untwisted one by $RG$.
\end{notation}

In the sequel, we reserve the name group rings for untwisted group rings.

Since all our rings are unital, we require ring homomorphisms to respect units.

\begin{example}\label{eg. twisted group ring}
Let $\phi\colon G\rightarrow H$ be a group homomorphism with kernel $K$. We choose a set-theoretic section $s\colon H\rightarrow G$, i.e., a map between the underlying sets such that $\phi\circ s=\mathrm{id}_H$. Let $(\Z K) H$ be the twisted group ring with structure functions $c_h(r)=s(h)rs(h)^{-1}$ and $\tau(h,h')=s(h)s(h')s(hh')^{-1}$. The untwisted group ring $\Z G$ is then isomorphic to the twisted group ring $(\Z K) H$ via the map
\[g\mapsto (g(s\circ \phi)(g)^{-1})\ast \phi(g).\]
\end{example}

\subsection{Ore localization}
The notion of the Ore localization is a generalization of the classical notation of the field of fractions of an integral domain. We will use this notion to rationalize a given agrarian map in Section \ref{sec. rationalization}.

Let $R$ be a ring and let $T\subset R$ be a subset of non-zero divisors of $R$. We say $T$ satisfies the \textit{left Ore condition} if for all $r\in R$ and $t\in T$, there exist $r_1\in R,t_1\in T$ such that $r_1t=t_1r$. The \textit{(left) Ore localization of $R$ with respect to $T$} is
\[\ore(R,T)=\{t^{-1}r\mid t\in T, r\in R\}.\]

If $R$ has no non-trivial zero divisor, $T$ is the set of non-zero elements of $R$, and $T$ satisfies the left Ore condition, then we briefly say that $R$ \textit{satisfies the Ore condition} and call the Ore localization of $R$ with respect to $T$ the \textit{Ore localization} of $R$; we denote it by $\ore(R)$, and note that it is a skew field.

\begin{example}\label{eg. D(X)}
Let $D$ be a skew field and let $H$ be a torsion-free amenable group. Every twisted group ring $DH$ that is a domain satisfies the Ore condition -- this follows from \cite{tamari1957refined}, and in this form is stated for example in \cite{kielak2018bieri}*{Theorem 2.14}.
\end{example}

\subsection{Laurent power series and orders}\label{sec. laurent}
Let $R$ be a ring, let $\alpha$ be an automorphism of $R$, and let $t\not\in D$ be a symbol. The ring of \textit{twisted Laurent power series} in $t$ with coefficients in $R$ is the set
\[R(\!(t)\!)=\left\{\sum^{\infty}_{i=k}r_it^i\mid k\in\Z\right\}.\]
Our convention is $r_i=0$ for $i<k$.

The multiplication on $R(\!(t)\!)$ is given by the convention
\[t\cdot r=\alpha(r)\cdot t.\]
With this multiplication and the obvious summation $R(\!(t)\!)$ is a ring. $\alpha$ is called the \textit{twisting structure} of $R(\!(t)\!)$. If $\alpha=\mathrm{id}_R$ then $R(\!(t)\!)$ will be called the \textit{ring of (untwisted) Laurent power series} and $t$ will be called a \textit{central variable}.

For the rest of this subsection we restrict ourselves to the case where $R=D$ is a skew field and $\alpha=\mathrm{id}_D$, in which case $D(\!(t)\!)$ is a skew field.

For each $x= \sum^{\infty}_{i=k}d_i t^i \in D(\!(t)\!)$, if $d_k \neq 0$, then we define the \textit{order} of $t$ in $x$ as
\[\ord_t x = k;\]
and if $x=0$, then $\ord_t x = \infty$ by definition. It is easy to check that if $x,y\in D(\!(t)\!)\smallsetminus\{0\}$, then
\[\ord_t (xy)=\ord_t x + \ord_t y.\]
So $\ord_t$ restricts to a homomorphism $\ord_t \colon D(\!(t)\!)_{\mathrm{ab}}^{\times}\rightarrow \Z$, where $D(\!(t)\!)_{\mathrm{ab}}^{\times}$ denotes the abelianization of the group of units $D(\!(t)\!)^\times$ of $D(\!(t)\!)$. Taking $0$ into account, we can also view $\ord_t$ as a semi-group homomorphism 
\[\ord_t\colon D(\!(t)\!)_{\mathrm{ab}}^{\times} \sqcup\{0\}\rightarrow \Z\sqcup\{\infty\}.\]
Here our convention is $\infty+n=\infty$ for all $n\in\Z \sqcup \{\infty\}$.

\subsection{Degree of rational functions}\label{sec. order}
The notion of the degree serves as a convenient way to compute the Thurston norm and, more generally, the agrarian norm. Let $D$ be a skew field, let $H$ be a finite-rank free abelian group with basis $X$, and let $DH$ be the untwisted group ring. Note that $DH$ is a domain and satisfies the Ore condition by Example \ref{eg. D(X)}.

Let $t\in X$. We would like to define the $t$-degree function in this subsection, which depends on the choice of $X$. We thus use the following convention to emphasize the role played by $X$. 

\begin{notation}
We will denote $DH$ by $D[X^{\pm}]$ and the Ore localization $\ore(DH)$ by $D(X)$.
\end{notation}

For each $p\in D[X^{\pm}]\smallsetminus\{0\}$, we defined the \textit{degree} of $t$ in $p$, denoted $\deg_t p$, as the maximal power of $t$ in $p$ minus the minimal power of $t$ in $p$. Note that $\deg_t$ is not an extension of the usual notion of the degree of a polynomial.  For $p=0$, we define $\deg_t p=-\infty$. The degree function can be extended to $D(X)$: for every element $f\in D(X)$, the \textit{degree} of $t$ in $f$, denoted $\deg_t f$, is
\[\deg_t f = \deg_t p - \deg_t q,\]
where $p,q\in DH,q\neq 0$ and $q^{-1}p=f$. Here, our convention is $-\infty+n=-\infty$ for all $n\in\Z$. The well-definedness of $\deg_t f$ is an easy exercise using the Ore condition.

The order function $\ord_t$ mentioned in Section \ref{sec. laurent} can be defined for elements of $D(X)$ in the following way. First, let $K$ be the subgroup of $H$ generated by $X\smallsetminus\{t\}$ and let $E=\ore(DK)$. Then we have a natural map $\alpha\colon D(X)\hookrightarrow E(\!(t)\!)$ by expanding every rational function into a Laurent power series: Let $f\in E[t]\smallsetminus\{0\}\subset D(X)$. Factorize $f$ as $f=d t^k\cdot (1+\sum^{\ell}_{i=1}d_i  t^i)$, where $d,d_i\in E,d\neq 0$. Define
\[\alpha(f^{-1})=\left(1+\sum^{\infty}_{j=1}\left(-\sum^{\ell}_{i=1}d_i  t^i\right)^{j}\right)\cdot  d^{-1}t^{-k}.\]
Every element of $D(X)$ can be written as a fraction $q^{-1}p$ with $p,q\in E[t],q\neq 0$. Define
\[\alpha(q^{-1}p)=\alpha(q^{-1})\cdot p.\]
That $\alpha$ is well-defined follows from the universal property of the Ore localization.

The function $\alpha$ embeds $D(X)$ into $E(\!(t)\!)$ as a subfield, and so we can view each element $f\in D(X)$ as an element of $E(\!(t)\!)$ and compute $\ord_t f$. Once again, this $t$-order depends on the basis $X$.

There is a ring homomorphism $\beta\colon E[t^{\pm}]\rightarrow E[t^{\pm}]$ such that $\beta(e  t^k)=e  t^{-k}$ for all $e\in E,k\in\Z$. For $f\in D(X)$, let $p,q\in D[X]$ such that $f=q^{-1}p$ and define
\[\beta(f)=\beta(q)^{-1}\beta(p).\]
That $\beta$ is well-defined follows from the universal property of the Ore localization.

Thus, $\beta$ extends to a ring homomorphism (still denoted by) $\beta\colon E(t)\rightarrow E(t)$, which can be easily seen to be a ring automorphism. We have
\begin{equation}\label{eq. degree and order}
    \deg_t f=-\ord_t f-\ord_t \beta(f).
\end{equation}

Just like $\ord_t$, the function $\deg_t$ descends to a semi-group homomorphism
\[\deg_t\colon D(X)^{\times}_{\mathrm{ab}}\sqcup\{0\}\rightarrow \Z\sqcup\{-\infty\}.\]

\subsection{Dieudonn\'{e} determinant}

The Dieudonn\'{e} determinant is a generalization to skew fields of the classical notion of determinant over a commutative field. It is indispensable in our definition of agrarian torsion and thus we recall its definition here. Let $D$ be a skew field and let $A=(A_{ij})$ be an $n\times n$-matrix over $D$. The \textit{canonical representative of the Dieudonn\'{e} determinant} $\det^c_D A\in D$ is defined inductively as follows:
\begin{enumerate}
    \item[(1)] If $n=1$, then $\det^c_D A=a_{11}$.
    \item[(2)] If the last row of $A$ consists of zeros only, then $\det^c_D A=0$.
    \item[(3)] If $a_{nn}\neq 0$, then we form the $(n-1)\times (n-1)$-matrix $A'=(a'_{ij})$ by setting $a'_{ij}=a_{ij}-a_{in}a^{-1}_{nn}a_{ij}$ and declare $\det^c_D A=\det^c_D A'\cdot a_{nn}$.
    \item[(4)] Otherwise, let $j<n$ be maximal such that $a_{nj}\neq 0$. Let $A'$ be obtained from $A$ by interchanging rows $j$ and $n$. Then set $\det^c_D A=-\det^c_D A'$.
\end{enumerate}
The \textit{Dieudonn\'{e} determinant} $\det_D A$ of $A$ is defined to be the image of $\det^c_D A$ in $D^{\times}_{\mathrm{ab}}\sqcup\{0\}$, where $D^{\times}_{\mathrm{ab}}$ is the abelianization of the multiplicative group of $D$.

\begin{remark}
The Dieudonn\'{e} determinant satisfies the following:
\begin{enumerate}[label=(\roman*)]
    \item\label{item. d-determinant 1} $\Det_D (AB)=\Det_D A\cdot \Det_D B$ for all square matrices $A,B$ of the same dimension.
    \item\label{item. d-determinant 2} If $A'$ is obtained from $A$ by adding a multiple of a row to another row, then $\det_D A'=\det_D A$.
    \item\label{item. d-determinant 3} If $A$ is upper-triangular, then $\det_D A$ is the image of $\prod^n_{i=1} a_{ii}$ in $D^{\times}_{\mathrm{ab}}\sqcup\{0\}$.
    \item\label{item. d-determinant 4} $A$ is invertible over $D$ if and only if $\Det_D(A) \neq 0$.
\end{enumerate}
\end{remark}

\subsection{Polytope group and polytope homomorphism}\label{sec. rational function}

In this subsection, we recall the notion of the polytope homomorphism, which is used later to construct the agrarian polytope. Let $V$ be an $\R$-vector space.

\begin{definition}
A \textit{polytope} in $V$ is the convex hull of finitely many points of $V$. Given two polytopes $P,Q$, the \textit{Minkowski sum} of $P,Q$, denoted $P+Q$, is the polytope
\[\{p+q\mid p\in P, q\in Q\}.\]
\end{definition}

Let $H$ be a finite-rank free abelian group. Below, we take the vector space $V$ to be $H_1(H,\R)\cong H\otimes_{\Z}\R$.

\begin{definition}
A polytope in $H_1(H,\R)$ is called \textit{integral} if the convex hull of some finite subset of the lattice $H\subset H_1(H,\R)\cong H\otimes_{\Z}\R$. The \textit{polytope group} $\p(H)$ of $H$ is the abelian group generated by formal differences $P-Q$ of non-empty integral polytopes $P,Q$ in $H_1(H,\R)$ with addition $(P-Q)+(P'-Q')=(P+P')-(Q+Q')$ and relations $P-Q=P'-Q'$ if $P+P'=Q+Q'$.

The unit of the group is the one-vertex polytope $\{0\}$, which we will denote by $0$. Also, instead of writing $P - 0$ for a polytope $P$ we will simply write $P$, and every element of $\p(H)$ of this form will be referred to as a \emph{single} polytope.
\end{definition}

Let $D$ be a skew field and let $DH$ be the (untwisted) group ring.

\begin{definition}
The \textit{Newton polytope} $P(p)$ of an element $p=\sum_{h\in H}d_h h\in DH$ is the convex hull of the \textit{support} $\supp(p)=\{h\in H\mid d_h\neq 0\}$ in $H_1(H,\R)$.
\end{definition}

\begin{definition}
The group homomorphism
\[P\colon  (\ore(DH))^{\times}_{\mathrm{ab}}\rightarrow \p(H),~P(q^{-1}p)=P(p)-P(q)\]
is called the \textit{polytope homomorphism} of $DH$.
\end{definition}
The well-definedness of $P$ is an easy exercise using the Ore condition. That $P$ is a group homomorphism is proved in \cite{kielak2018bieri}*{Lemma 3.12}.

The following is a special case of \cite{kielak2018bieri}*{Theorem 3.14}.

\begin{theorem}
\label{thm. single polytope}
Let $A$ be a square matrix over $DH$ with $\Det_{\ore(DH)}(A)\neq 0$. Then $P(\Det_{\ore(DH)}(A))$ is a single polytope.
\end{theorem}

\subsection{A lemma about matrices over skew fields}
In this subsection, we prove a technical lemma about matrices, which will be used several times in the sequel. Let $D$ be a skew field and let $t$ be a central variable. Let $D(\!(t)\!)$ be the ring of Laurent power series of $t$ with coefficient in $D$. Consider a matrix $M$ of the form
\[M=\mathrm{Id}+N\cdot t,\]
where $N=(n_{ij})\in M_n(D(\!(t)\!))$ and each entry of $N$ has $t$-order at least $0$, i.e., $\ord_t n_{ij}\geqslant 0$ for all $i,j$. In order to compute $\det_{D(\!(t)\!)}M$, we use the following process  to turn $M$ into an upper-triangular matrix:

\begin{enumerate}
    \item[$(\ast)$] This process consists of $n$ steps. In the ${i}^{\mathrm{th}}$ step, we use elementary row operations to eliminate all $(j,i)$-entries with $j>i$. In other words, for each $j>i$, we add a suitable left multiple of the ${i}^{\mathrm{th}}$ row to the ${j}^{\mathrm{th}}$ row so that the resulting matrix has $0$ as its $(j,i)$-entry.
\end{enumerate}

Note that in order to carry out the $(\ast)$ process, at the $i^{\mathrm{th}}$ step we need the resulting matrix to have a non-zero $(i,i)$-entry, which is a priori unclear and thus it is a priori unclear whether the $n$ steps in process $(\ast)$ can all be carried out. The following lemma affirms the feasibility of $(\ast)$.

\begin{lemma}\label{lem. matrix reduction}
After each step of process $(\ast)$, we will get a matrix $M'=\mathrm{Id}+N'\cdot t$, where each entry of $N'$ has $t$-order at least $0$. In particular, $M'$ has non-zero diagonal entries, and thus all $n$ steps of process $(\ast)$ can be carried out and the matrix obtained from process $(\ast)$, say $\overline M=(\overline m_{ij})$, is upper-triangular and we have for all $i$, $\overline m_{ii}=1+\overline n_{ii}\cdot t$, where $\ord_t \overline n_{ii}\geqslant 0$. In particular, $M$ is invertible.
\end{lemma}

\begin{proof}
First, after step $0$ (that is, before we do anything), the lemma is clear.

Suppose after step $i-1$, we get a matrix $M'=(m'_{jk})=\mathrm{Id}+N'\cdot t$, where each entry of $N'$ has $t$-order at least $0$. Consider step $i$. Let $j>i$. Step $i$ turns each entry $m'_{jk}$ into
\[m''_{jk}=m'_{jk}-m'_{ji}(m'_{ii})^{-1}m'_{ik}.\]

Note that $\ord_t (m'_{ji}(m'_{ii})^{-1}m'_{ik})\geqslant 1$. If $j\neq k$ then $\ord_t m'_{jk}\geqslant 1$ and thus $\ord_t m''_{jk}\geqslant 1$. If $j=k$ then $m'_{jk}=1+n'_{jk}\cdot t$ for some $n'_{jk}\in D(\!(t)\!)$ with $\ord_t n'_{jk}\geqslant 0$, and thus $m''_{jk}=1+n''_{jk}\cdot t$ for some $n''_{jk}\in D(\!(t)\!)$ with $\ord_t n''_{jk}\geqslant 0$.
\end{proof}

\subsection{Behavior of the degree function under representations of skew fields}\label{sec. matrix reduction}
Let $D$ be a skew field and let $t$ be a central variable. Let $n\in \mathbb{N}^+$ and suppose that there is a ring homomorphism $\sigma\colon D\rightarrow M_n(D)$. We extend $\sigma$ to a ring homomorphism (still denoted by) $\sigma\colon D(\!( t )\!) \rightarrow M_n(D(\!( t )\!))$ by
\[\sigma\left( \sum^{\infty}_{i=k} d_it^i \right)=\sum^{\infty}_{i=k}\sigma(d)\cdot t^i.\]

The goal of this subsection is:

\begin{lemma}\label{lem. extension and degree}
For every $A \in M_k(D [t^\pm])$ we have
\[\deg_t\Det_{D(t)}\sigma(A)=n\cdot \deg_t\Det_{D(t)}A,\]
where $\sigma(A)$ is the matrix obtained by applying $\sigma$ to every entry of $A$.
\end{lemma}

Before the proof of Lemma \ref{lem. extension and degree}, we state a useful corollary.

\begin{corollary}\label{cor. ring homo extension}
The representation $\sigma$ extends to a ring homomorphism $D(t)\rightarrow M_n(D(t))$.
\end{corollary}

\begin{proof}
By Lemma \ref{lem. extension and degree} applied with $k=1$, the matrix $\sigma(p)$ is invertible in $M_n(D(t))$ for every non-zero $p\in D[t]$. By the universal property of the Ore localization, $\sigma$ extends to a ring homomorphism 
\[D(t)=\ore(D[t])\rightarrow \ore(M_n(D[t]),T)=M_n(D(t)),\]
where $T\subset M_n(D[t])$ is the set of non-zero divisors. Here, the last equality follows from the universal property of the Ore localization and the following two observations about the natural inclusion $M_n(D[t])\hookrightarrow M_n(D(t))$:
\begin{enumerate}
    \item[(i)] Every matrix $A\in T$ is invertible in $M_n(D(t))$. Indeed, suppose $A$ is not invertible over $M_n(D(t))$. Then there is a non-zero matrix $B\in M_n(D(t))$ such that $B\cdot A=0$. Using the Ore condition, one can find a non-zero element $p\in D[t]$ such that $p\cdot B$ is a matrix over $D[t]$. But then $A$ is a non-trivial zero divisor as $p\cdot B\cdot A=0$, a contradiction.
    
    \item[(ii)] Every ring homomorphism $\phi\colon M_n(D[t])\rightarrow R$ that maps every element of $T$ to an invertible element uniquely extends to a ring homomorphism $M_n(D(t))\rightarrow R$. Indeed, for every matrix $A\in M_n(D(t))$, there exists a non-zero element $p\in D[t]$ such that $p\cdot A$ is a matrix over $D[t]$. Note that $p\cdot \mathrm{Id}\in T$ and thus $\phi(p\cdot \mathrm{Id})$ is invertible in $R$. So if $\phi$ can be extended to $M_n(D(t))$, it has to map $A$ to $\phi(p\cdot \mathrm{Id})^{-1}\cdot \phi(p\cdot A)$, which shows uniqueness. To extend $\phi$, define $\phi(A)=\phi(p\cdot \mathrm{Id})^{-1}\cdot \phi(p\cdot A)$. We have to check that this is well defined. Assume that there is another non-zero element $q\in D[t]$ such that $q\cdot A\in M_n(D[t])$. Then by the Ore condition there exist $r,s\in D[t], s\neq 0$ such that $rp=sq$. In particular we have $rp\neq 0$ and also $r\neq 0$. We have
    \begin{align*}
        \phi(p\cdot\mathrm{Id})^{-1}\cdot\phi(p\cdot A)&=\phi(p\cdot\mathrm{Id})^{-1}\cdot \phi(r\cdot\mathrm{Id})^{-1}\cdot\phi(rp\cdot A)\\
        &=\phi(rp\cdot\mathrm{Id})^{-1}\cdot \phi(rp\cdot A)\\
        &=\phi(sq\cdot\mathrm{Id})^{-1}\cdot \phi(sq\cdot A)\\
        &=\phi(q\cdot\mathrm{Id})^{-1}\cdot \phi(s\cdot\mathrm{Id})^{-1}\cdot\phi(sq\cdot A)\\
        &=\phi(q\cdot\mathrm{Id})^{-1}\cdot\phi(q\cdot A).\qedhere
    \end{align*}
\end{enumerate}
\end{proof}

The lemma below is the first step towards proving Lemma \ref{lem. extension and degree}.

\begin{lemma}\label{lem. extension and order}
For all $z\in D(\!( t )\!)$,
\[ \ord_t \Det_{D(\!( t )\!)} \sigma(z) = n\cdot \ord_t z.\]
\end{lemma}

\begin{proof}
The lemma is trivial if $z=0$. So let $z=\sum^{\infty}_{i=\ell}d_i t^i\in D(\!( t )\!)\smallsetminus\{0\}$, where $d_i\in D$ and $d_{\ell}\neq 0$. Then
\[z=d_{\ell}\cdot t^{\ell}\cdot\left(\sum^{\infty}_{i=\ell}d^{-1}_{\ell} d_i t^{i-\ell}\right).\]

Thus
\[
    \ord_t\Det_{D(\!( t )\!)}\sigma(z)=\ord_t\Det_{D(\!( t )\!)}\sigma(d_{\ell})+\ord_t\Det_{D(\!( t )\!)}\sigma(t^{\ell})+ \ord_t\Det_{D(\!( t )\!)}\sigma\left(\sum^{\infty}_{i=\ell}d^{-1}_{\ell} d_i t^{i-\ell}\right).
\]

Since $d_{\ell}\neq 0$, $\sigma(d_{\ell})$ is an invertible matrix over $D$. Thus,
\[\ord_t\Det_{D(\!( t )\!)}(\sigma(d_{\ell}))=0.\]
By Lemma \ref{lem. matrix reduction}, we have
\[\ord_t\Det_{D(\!( t )\!)}\sigma\left(\sum^{\infty}_{i=\ell}d^{-1}_{\ell} d_i t^{i-\ell}\right)=0.\]
Note also that
\[\ord_t\Det_{D(\!( t )\!)}(\mathrm{Id}_n\cdot t^{\ell})=n\ell.\]
We thus have $\ord_t\Det_{D(\!( t )\!)}\sigma(z)=n\ell=n\cdot\ord_t z$, as desired.
\end{proof}

\begin{proof}[Proof of Lemma \ref{lem. extension and degree}]
First consider $z\in D(t)$. Let $\alpha\colon D(t)\rightarrow D(\!( t )\!)$ be the embedding given in Section \ref{sec. order} and let $z'=\Det^c_{D(t)}\sigma(z)$ be the canonical representative of the Dieudonn\'{e} determinant. Then
\[\ord_t\Det_{D(t)}\sigma(z)=\ord_t z'=\ord_t \alpha(z')=\ord_t\Det_{D(\!( t )\!)}\sigma(z)=n\cdot \ord_t z,\]
where the last equality follows from Lemma \ref{lem. extension and order}.

Let $\beta\colon D(t) \rightarrow D(t)$ be the ring automorphism constructed in Section \ref{sec. order}. We also view $\beta$ as an automorphism of the semi-group $(D(t))^{\times}_{\mathrm{ab}}\sqcup\{0\}$. We have
\[
    \ord_t\beta(\Det_{D(t)}\sigma(z))=\ord_t\Det_{D(t)}\beta(\sigma(z))
    =\ord_t\Det_{D(t)}\sigma(\beta(z))=n\cdot \ord_t\beta(z),
\]
where the third equality follows from Lemma \ref{lem. extension and order}.

Equation \eqref{eq. degree and order} then implies
\begin{equation}\label{eq. 1x1}
    \deg_t\Det_{D(t)}\sigma(z)=n\cdot \deg_t z.
\end{equation}

Now consider the matrix $A$. By elementary row operations over $D(t)$ we can turn $A$ into an upper-triangular matrix over $D(t)$. In more details, there are elementary matrices $U_1,\dots, U_\kappa\in M_k\big(D(t)\big)$ whose diagonal entries are all $\pm 1$ such that $B=(\prod^\kappa_{i=1}U_i)A\in D(t)$ is an upper-triangular matrix. So
\[\deg_t\Det_{D(t)} B=\deg_t(\pm\Det_{D(t)}A)=\deg_t\Det_{D(t)} A.\]

The matrix $\sigma(B)$ is a block-wise upper-triangular matrix. Note that $\Det_{D(t)}\sigma(U_i)=\pm 1$ for all $i$. Thus
\[\Det_{D(t)}\sigma(B)=\left(\prod^k_{i=1}\Det_{D(t)}\sigma(U_i)\right)\cdot\Det_{D(t)}\sigma(A)=\pm\Det_{D(t)}\sigma(A).\] 
Since $\Det^c_{D(t)} B$ is the product of its diagonal entries, from equation \eqref{eq. 1x1} we see that 
\[\deg_t\Det_{D(t)}\sigma(B)=n\cdot \deg_t\Det_{D(t)}B=n\cdot \deg_t\Det_{D(t)} A.\qedhere\]
\end{proof}

\subsection{Locally indicable groups and Linnell skew fields}
A group $G$ is \textit{locally indicable} if every non-trivial finitely generated subgroup of $G$ admits an epimorphism onto $\Z$.

\begin{example}\label{eg. locally indicable}
All free groups are locally indicable. More generally, except for the fundamental group of the projective plane $P^2$, all surface groups are locally indicable. Indeed, suppose $S\neq P^2$ is a surface. If $S$ is not closed, then $\pi_1(S)$ is free and thus is locally indicable. If $S=S^2$ is the $2$-sphere, then $\pi_1(S)$ is of course locally indicable. In all other cases $b_1(S)\neq 0$ and thus $\pi_1(S)$ has a surjection onto $\Z$. The kernel $K$ of this surjection corresponds to an infinite cyclic cover of $S$ and thus is a free group. So $\pi_1(S)$ is a semi-direct product $\pi_1(S)=K\rtimes \Z$ with $K$ free, and thus is locally indicable.
\end{example}

To obtain more examples of locally indicable groups we briefly recall some notions from the theory of $3$-manifolds. For details the reader is referred to the book \cites{afw2015three}. Let $M$ be a compact connected orientable irreducible $3$-manifold with empty or toroidal boundary. $M$ is called \textit{non-positively curved} if there is a Riemannian metric on the interior of $M$ with non-positive sectional curvature. $M$ is called \textit{geometric} if $M$ supports one of the geometries $S^3,S^2\times\mathbb{R}^1,\mathbb{R}^3,\mathrm{NIL},\mathrm{SOL},\widetilde{\mathrm{SL}_2(\mathbb{R})},\mathbb{H}^3$, $\mathbb{H}^2\times\mathbb{R}$. Note that if $M$ has non-empty boundary then $M$ supports either $\mathbb{H}^3$ or $\mathbb{H}^2\times\mathbb{R}$

\begin{lemma}\label{lem. eg of locally indicable}
Suppose that $M$ is a compact connected orientable irreducible $3$-manifold with empty or toroidal boundary. If $M$ either has positive virtual first Betti number or is non-positively curved or is geometric, then $\pi_1(M)$ is virtually locally indicable, i.e., $\pi_1(M)$ has a finite-index locally indicable subgroup.
\end{lemma}

\begin{proof}
    First suppose that $M$ has positive virtual first Betti number. By passing to a finite-sheeted cover of $M$, we may assume $b_1(M)>0$. Let $H\leqslant \pi_1(M)$ be a finitely generated non-trivial subgroup. Then if $H$ is of finite index, then $b_1(H)\geqslant b_1(M)>0$. If $H$ is of infinite index, then the proof of \cite{howie1982locally}*{Theorem 6.1} (see also \cite{howie1985band}*{Lemma 2}) shows that $b_1(H)\geqslant 1$. So in any case $H$ has a surjection onto $\Z$.

    Suppose instead that $M$ is non-positively curved. Then the works of Agol, Duchamp, Gruenberg, Haglund, Kahn, Krob, Markovi\'c, Liu, Perelman, Przytycki, Rhemtulla, and Wise (see \cite{afw2015three}*{(G.30)} for an explanation) imply that $\pi_1(M)$ is virtually bi-orderable and thus is virtually locally indicable \cites{levi1942ordered}.
    
    Suppose that $M$ is geometric. If $M$ supports one of the geometries $S^3,S^2\times\mathbb{R}^1,\mathbb{R}^3,\mathrm{NIL},\mathrm{SOL}$, then $M$ is closed and then by \cite{groves2012recognizing}*{Lemmata 8.1, 9.2, 10.1, 11.1}, $\pi_1(M)$ is either virtually free abelian or virtually $\Z\rtimes\Z$, and thus is virtually locally indicable.
    
    If $M$ supports the $\widetilde{\mathrm{SL}_2(\mathbb{R})}$ geometry then $\pi_1(M)$ is a semi-direct product $\pi_1(M)=\Z\rtimes F$ for some non-cyclic free group $F$ and thus is locally indicable.
    
    If $M$ supports the $\mathbb{H}^3$ geometry, then it is non-positively curved and thus by the second paragraph of this proof is virtually locally indicable.
    
    In $M$ supports the $\mathbb{H}^2\times \mathbb{R}$ geometry then $\pi_1(M)$ is virtually a product $\Z\times F$ where $F$ is a non-cyclic free group, and thus is virtually locally indicable.
\end{proof}

\begin{remark}
Note that if $M$ is a compact connected orientable irreducible $3$-manifold with empty or toroidal boundary and $M$ is not a closed graph manifold, then $M$ is non-positively curved, and thus $\pi_1(M)$ is virtually locally indicable. Indeed, the non-positive curvature of $M$ follows from the resolution of the Virtually Haken Conjecture \cites{agol2013virtual}, Thurston's Hyperbolication Theorem \cites{thurston1986hyperbolic} and the work of Leeb \cites{leeb1995three}.
\end{remark}

Let $G$ be a locally indicable group. By \cite{jaikin2020strong}*{Theorem 1.1}, $G$ satisfies the Atiyah conjecture over $\C$. Let $\D_G$ be the division closure of $\C G$ in $\mathcal{U}_G$, the algebra of affiliated operators of the group von Neumann algebra of $G$. Since $G$ is obviously torsion-free, $\D_ G$ is a skew field \cite{linnell1993division}*{Theorem 1.3} and is called the \textit{Linnell skew field} of $\C G$.

\subsection{Hughes-free skew fields}
The notion of Hughes-freeness was introduced by Hughes \cite{hughes1970division} in order to prove isomorphism between certain skew fields.

\begin{definition}\label{def. specialization}
Let $R$ be a ring.  An $R$\textit{-field} consists of a skew field $D$ and a ring homomorphism $\beta\colon R\rightarrow D$. The skew field $D$ is called an \textit{epic} $R$\textit{-field} if $D$ is the skew field generated by $\beta(R)$.
\end{definition}

\begin{definition}
Let $DG$ be a twisted group ring with $D$ a skew field and $G$ a locally indicable group. An epic $DG$-field $\beta\colon DG\rightarrow E$ is \textit{Hughes-free} if for every non-trivial finitely generated subgroup $H$ of $G$, every normal subgroup $N$ of $H$ with $H/N\cong \Z$, and every $h_1,\cdots, h_n\in H$ in distinct cosets of $N$ in $H$, the sum $E_N\beta(h_1)+\cdots+E_N\beta(h_n)$ is direct, where $E_N$ is the division closure of $\beta(DN)$ in $E$, and $DN$ is the subring of $DG$ generated by $D$ and $N$.
\end{definition}

\begin{example}\label{eg. hughes free}
Let $G$ be a locally indicable group. Then its Linnell skew field is a Hughes-free $\C G$-field \cite{jaikin2020strong}*{Corollary 6.2}.
\end{example}

\begin{theorem}[\cites{hughes1970division}]\label{thm. hughes}
Let $DG$ be a twisted group ring with $D$ a skew field and $G$ a locally indicable group. Let $\beta\colon DG\rightarrow E,\beta'\colon DG\rightarrow E'$ be two Hughes-free $DG$-fields. Then there exists a ring isomorphism $\alpha\colon E\rightarrow E'$ such that $\beta'=\alpha\circ \beta$. 
\end{theorem}

\subsection{Specialization, universality and Lewin groups}
Let $R$ be a ring.

\begin{definition}
Given two epic $R$-fields $\beta \colon R\rightarrow D$ and $\beta'\colon R\rightarrow D'$, a \textit{specialisation} of $D$ to $D'$ with respect to $R$ is a pair $(S,\alpha)$ where $S$ is a subring of $D$ containing $\im\beta$, the map $\alpha\colon S\rightarrow D'$ is a ring homomorphism with $\alpha\circ \beta=\beta'$, and every element in $S$ not mapped to $0$ by $\alpha$ is invertible in $S$. 
\end{definition}

\begin{definition}
An epic $R$-field $\beta\colon R\rightarrow D$ is called the \textit{universal} $R$\textit{-field} if for every epic $R$-field $D'$ there is a specialization of $D$ to $D'$ with respect to $R$. If in addition the map $R\rightarrow D$ is injective, then $D$ is called \textit{the universal field of fractions} of $R$.
\end{definition}

\begin{definition}
A group $G$ is \textit{Lewin} if for every twisted group ring $DG$ with $D$ a skew field, there is a Hughes-free universal $DG$-field.
\end{definition}

Let $G$ be a finitely generated Lewin group. By \cite{jaikin2020universality}*{Proposition 4.1}, $G$ is locally indicable and thus there is a natural embedding $\tau\colon \C G \rightarrow \D_G$ of $\C G$ into its Linnell skew field $\D_G$. 

Let $q\colon G\twoheadrightarrow G_{\mathrm{fab}}$ be the natural quotient homomorphism of $G$ onto its maximal free abelian quotient $G_{\mathrm{fab}}$. Let $X$ be a basis of $G_{\mathrm{fab}}$ and let $x\in X$. Consider the group ring $\D_G G_{\mathrm{fab}}$. Since we are interested in computing $\deg_x$, the difference between the highest and the lowest power of $x$ with respect to the basis $X$, we will denote $\ore(\D_G G_{\mathrm{fab}})$ by $\D_G(X)$ to emphasize the role played by $X$.

Let $\sigma\colon G\rightarrow \GL_n(\C)$ be a complex representation of $G$ of finite dimension $n$. For reasons that will be clear later, we are interested in the representations
\begin{align*}
    \sigma\otimes_{\Z} q \colon& G\rightarrow \GL_n(\C(X)),~~ (\sigma\otimes_{\Z} q)(g)=\sigma(g)q(g),\\
   \sigma\otimes_{\C}\tau\otimes_{\Z} q \colon& G\rightarrow \GL_n(\D_G(X)),~~(\sigma\otimes_{\C}\tau\otimes_{\Z} q)(g)=\sigma(g)\tau(g)q(g).
\end{align*}

Let $M$ be an invertible square matrix over $\Z G$. By applying $\sigma\otimes_{\Z} q$ and ${\sigma\otimes_{\C}\tau\otimes_{\Z} q}$ to every entry of $M$ we obtain matrices $(\sigma\otimes_{\Z} q)(M)$ and $(\sigma\otimes_{\C}\tau\otimes_{\Z} q)(M)$, respectively.

\begin{lemma}\label{lem. key lem}
We have the inequality
\[\deg_x\Det_{\C(X)}(\sigma\otimes_{\Z}q)(M)\leqslant \deg_x\Det_{\D_G(X)} (\sigma\otimes_{\C}\tau\otimes_{\Z} q) (M).\]
\end{lemma}

\begin{proof}
We would like to apply \cite{funke2018alexander}*{Proposition 4.1}. Let $Y=X\smallsetminus\{x\}$. Consider the ring $R=\C G[Y^{\pm}]$. The ring $\C G[X^{\pm}]=R[x^{\pm}]$ is the ring of Laurent polynomials over $R$. Consider two $R$-fields 
\begin{align*}
    \beta \colon & R\rightarrow \D_G(Y),~~\beta(g)=\tau(g),\beta(y)=y \text{ for all }g\in G,y\in Y,\\
    \beta' \colon & R\rightarrow \C(Y),~~\beta'(g)=1, \beta'(y)=y \text{ for all }g\in G,y\in Y.
\end{align*}

Consider the representation
\[\sigma\otimes_{\C}\mathrm{id}_{\C G}\otimes_{\Z}q\colon G \rightarrow \GL_n(R[x^{\pm}]).\]

As above, by applying $\sigma\otimes_{\C}\mathrm{id}_{\C G}\otimes_{\Z}q$ to every entry of $M$ we get a matrix $(\sigma\otimes_{\C}\mathrm{id}_{\C G}\otimes_{\Z}q)(M)$. Note that $\beta'$ can be extended to a map from $R[x^{\pm}]$ to $\C(X)$ by setting $\beta'(x)=x$. With this convention we can then apply $\beta'$ to each entry of the matrix $(\sigma\otimes_{\C}\mathrm{id}_{\C G}\otimes_{\Z}q)(M)$ to get a square matrix $\beta'((\sigma\otimes_{\C}\mathrm{id}_{\C G}\otimes_{\Z}q)(M))$ over $\C(X)$. Note that
\[(\sigma\otimes_{\Z}q)(M)=\beta'((\sigma\otimes_{\C}\mathrm{id}_{\C G}\otimes_{\Z}q)(M)).\]
Similarly, we can extend $\beta$ to a map from $R[x^{\pm}]$ to $\D_G(X)$ by setting $\beta(x)=x$. We have
\[(\sigma\otimes_{\C}\tau\otimes_{\Z} q) (M)=\beta((\sigma\otimes_{\C}\mathrm{id}_{\C G}\otimes_{\Z}q)(M)).\]

By \cite{jaikin2020universality}*{Theorem 3.7}, $\D_G$ is the universal field of fractions of $\C G$. Therefore, there exists a specialization of $\D_G(Y)$ to $\C(Y)$ with respect to $R$. The desired result thus follows from \cite{funke2018alexander}*{Proposition 4.1}.
\end{proof}

\subsection{Rational semirings}
In this and the next two subsections we recall results about rational semirings that are necessary in the study of twisted $\ell^2$-Betti numbers. By a \textit{semiring} $R$ we mean a set together with an associative addition and an associative multiplication with identity element $1_R$ which is distributive over the addition. Let $U$ be a group and let $R$ be a semiring. We say that $R$ is a \textit{rational $U$-semiring} if
\begin{enumerate}
    \item[(i)] There is a map $\diamond\colon R\rightarrow R, r\mapsto r^{\diamond}$, called the \textit{rational structure} on $R$.
    \item[(ii)] $R$ is a \textit{$U$-biset}, i.e., $U$ acts on both sides of $R$ in a compatible way: $(ur)v=u(rv)$ for all $u,v\in U,r\in R$.
    \item[(iii)] For every $u,v\in U$ and $r\in R$, $(urv)^{\diamond}=v^{-1}r^{\diamond}u^{-1}$.
\end{enumerate}

\begin{example}\label{eg. rational}
Let $G$ be a group and let $R$ be a ring with a ring homomorphism $\sigma\colon \C G\rightarrow R$. Then $R$ is a $\C^{\times} G$-biset with the action given by
\[(c_1g_1,c_2g_2,r)\mapsto c_1\sigma(g_1)\cdot r\cdot c_2\sigma(g_2)\]
for all $c_1,c_2\in \C^{\times},g_1,g_2\in G,r\in R$.
Let $S$ be the division closure of $\sigma(\C G)$ in $R$. Then $S$ is a rational $\C^{\times}G$-semiring under the rational map given by the following: if $s\in S$ is invertible in $R$ then $s^{\diamond}=s^{-1}$; otherwise, $s^{\diamond}=0$.
\end{example}

A \textit{morphism of rational $U$-semirings} $\Phi\colon R_1\rightarrow R_2$ is a map such that
\begin{enumerate}
    \item[(i)] $\Phi(r+r')=\Phi(r)+\Phi(r')$;
    \item[(ii)] $\Phi(rr')=\Phi(r)\Phi(r')$ and $\Phi(1_{R_1})=1_{R_2}$;
    \item[(iii)] $\Phi(r^{\diamond})=\Phi(r)^{\diamond}$ for all $r\in R_1$;
    \item[(iv)] $\Phi(urv)=u\Phi(r)v$ for all $u,v\in U,r\in R_1$.
\end{enumerate}

Below, we recall the construction of the \textit{universal rational $U$-semiring} $\rat(U)$. It is characterized by the following universal property:

\begin{lemma}[\cite{dicks2004theorem}*{Lemma 4.7}]\label{lem. universal morphism}
If $R$ is a rational $U$-semiring, then there exists a unique morphism of rational $U$-semirings $\Phi\colon \rat(U)\rightarrow R$.
\end{lemma}

Before defining $\rat(U)$, we present some definitions and notation:

\begin{itemize}
    \item If $X$ is a set, then the free additive semigroup on $X$ is $\N X\smallsetminus\{0\}$. Here, our convention is $0\in\N$. Note that if $X$ is a multiplicative monoid with a $U$-biset structure, then $\N X\smallsetminus\{0\}$ is naturally a $U$-semiring.
    \item If $X$ is a $U$-biset, then $X^{\times^n_U}$ is the set of equivalence classes of words in $X$ of length $n$ with respect to the relation generated by
    \[x_1x_2\cdots(x_iu)x_{i+1}\cdots x_n\sim x_1x_2\cdots x_i(ux_{i+1})\cdots x_n\] 
    for all $x_1,x_2,\cdots,x_n\in X, u\in U$. The multiplicative free monoid on $X$ over $U$ is defined as
    \[U\natural X=\bigcup^{\infty}_{n=0}X^{\times^n_U}\]
    where by definition $X^{\times^0_U}=U$. Observe that $\N [U\natural X]\smallsetminus\{0\}$ is naturally a $U$-semiring.
    \item If $X$ is a $U$-biset, then $X^{\diamond}$ denotes a disjoint copy of $X$ together with a bijective map $X\rightarrow X^{\diamond},x\mapsto x^{\diamond}$, and a $U$-biset structure given by 
    \[ux^{\diamond}v=(v^{-1}xu^{-1})^{\diamond}\]
    for all $u,v\in U,x\in X$.
\end{itemize}

The \textit{universal rational $U$-semiring} is defined as follows:
\begin{itemize}
    \item First consider the $U$-semiring $\N U\smallsetminus\{0\}$ and set $X_0=\emptyset,X_1=(\N U\smallsetminus\{0\})^{\diamond}$.
    \item Suppose $n\geqslant 1$, $X_n$ is a $U$-biset and $X_{n-1}$ is a $U$-sub-biset of $X_n$. Consider the $U$-semiring $\N[U\natural X_n]\smallsetminus\{0\}$ and the $U$-sub-biset $\N[U\natural X_n]\smallsetminus \N[U\natural X_{n-1}]$. Define
    \[X_{n+1}=(\N[U\natural X_n]\smallsetminus \N[U\natural X_{n-1}])^{\diamond}\cup X_n.\]
    \item Then $X=\bigcup_{n\geqslant 0} X_n$ is a $U$-biset. Let
    \[\rat(U)=\N[U\natural X]\smallsetminus\{0\}.\]
\end{itemize}

For later reference, we note the following.

\begin{theorem}[\cite{dicks2004theorem}*{Lemma 5.4 and Theorem 5.7}]\label{thm. source}
If $\alpha\in \rat(U)$, then there exists a subgroup $\mathrm{source}(\alpha)\leqslant U$ with the following properties.
\begin{enumerate}[label=(\roman*)]
    \item\label{item. source 1} $\mathrm{source}(\alpha)$ is finitely generated and $\alpha\in \rat(\mathrm{source}(\alpha))\cdot U$.
    \item\label{item. source 2} If $V$ is a subgroup of $U$ such that $\alpha\in \rat(V)\cdot U$, then $\mathrm{source}(\alpha)\leqslant V$.
\end{enumerate}
\end{theorem}

An element $\alpha\in\rat(U)$ is called \textit{primitive} if $\alpha\in\rat(\mathrm{source}(\alpha))$.

\subsection{Trees and complexity}\label{sec. tree}

Let $\mathcal{T}$ be the set of all finite rooted trees up to isomorphism. Here we recall that $\mathcal{T}$ has a well-order satisfying certain properties and is a $U$-semiring for any group $U$. The order will be used to define a complexity on elements of $\rat(U)$.

Denote by $0_{\mathcal{T}}$ the one-vertex tree. If $0_{\mathcal{T}}\neq X\in\mathcal{T}$, denote by $\mathrm{fam}(X)$ the finite family of finite rooted trees obtained from $X$ by deleting the root and all incident edges, where the root of an element $Y\in \mathrm{fam}(X)$ is the unique vertex of $Y$ that is incident to the root of  $X$. We denote by $\exp(X)$ the tree obtained from $X$ by adding a new vertex which is declared to be the root of $\exp(X)$, and a new edge joining it to the root of $X$.

Let $X,Y\in\mathcal{T}$. The sum $X+Y\in\mathcal{T}$ is the rooted tree obtained by identifying the roots of $X,Y$ and declaring it to be the root of $X+Y$. The product $X\cdot Y$ is defined as follows: if one of $X,Y$ is $0_{\mathcal{T}}$, then $X\cdot Y=0_{\mathcal{T}}$ by definition; if $X,Y\neq 0_{\mathcal{T}}$, the product $X\cdot Y$ is obtained by adding pairwise elements of $\mathrm{fam}(X)$ with elements of $\mathrm{fam}(Y)$, and then connecting all the resulting finite rooted trees by adding a new vertex with incident edges to their roots, and declaring the new vertex to be the root of $X\cdot Y$, i.e.,
\[X\cdot Y=\sum_{\substack{X'\in\mathrm{fam}(X)\\ Y'\in\mathrm{fam}(Y)}}\exp(X'+Y').\]
The rational map of $\mathcal{T}$ is given by
\[X^{\diamond}=\exp^2(X).\]
The group $U$ acts on both sides of $\mathcal{T}$ by the trivial action. With these operations, $\mathcal{T}$ is a rational $U$-semiring.

Let $\mathcal{T}_n\subset\mathcal{T}$ be the subset consisting of all elements with at most $n$ edges. The following defines a well-order on $\mathcal{T}$ \cite{dicks2004theorem}*{Lemma 3.3}:
\begin{itemize}
    \item $0_{\mathcal{T}}$ is the least element of $\mathcal{T}$.
    \item Suppose that $\mathcal{T}_{n-1}$ has already been ordered for some $n\geqslant 1$. Let $X,Y\in\mathcal{T}_n\smallsetminus\{0_{\mathcal{T}}\}$. Let $\log(X)$ be the largest element of $\mathcal{T}_{n-1}$ in $\mathrm{fam}(X)$, so $\exp(\log(X))$ is a summand of $X$, and denote its complement by $X-\exp(\log(X))\subset \mathcal{T}_{n-1}$. Define $X>Y$ if either $\log(X)>\log(Y)$ or $\log(X)=\log(Y)$ and $X-\exp(\log(X))>Y-\exp(\log(Y))$.
\end{itemize}

By Lemma \ref{lem. universal morphism}, there is a unique map 
\[\tree\colon \rat(U)\cup\{0\}\rightarrow \mathcal{T}\]
that maps $0$ to $0_{\mathcal{T}}$. For $\alpha\in \rat(U)\cup\{0\}$, the image $\tree(\alpha)$ is called the \textit{complexity} of $\alpha$.

\begin{remark}\label{rm. complexity}
If $V\leqslant U$ is a subgroup  then by Lemma \ref{lem. universal morphism} there is a unique map
\[\tree_V\colon \rat(V)\cup \{0\}\rightarrow \mathcal{T}\]
that maps $0$ to $0_{\mathcal{T}}$. If $\alpha\in\rat(V)\cup\{0\}\subset\rat(U)\cup\{0\}$, then $\tree_V(\alpha)=\tree(\alpha)$ for all $\alpha\in \rat(V)\cup \{0\}$, i.e., the complexity of $\alpha$ does not depend on whether we consider $\alpha$ as an element of $\rat(V)\cup\{0\}$ or $\rat(U)\cup\{0\}$.
\end{remark}

\subsection{Tree complexity associated to groups}\label{sec. complexity}
Let $G$ be a locally indicable group, let $\tau\colon \C G \rightarrow \D_G$ be the natural embedding into the Linnell skew field, and let $\sigma\colon G \rightarrow \mathrm{GL}_n(\C)$ be a finite-dimensional representation.

For every subgroup $H\leqslant G$, we think of the Linnell skew field $\D_H$ as a subring of $\D_G$. Let $\widetilde\D_H$ be the division closure of $(\sigma\otimes_{\C}\tau)(\C H)$ in $M_n(\D_G)$. Note that by the definition of the division closure we have
\begin{equation}\label{eq. subring}
    \widetilde \D_H\leqslant M_n(\D_H).
\end{equation}

Think of $\sigma\otimes_{\C}\tau$ as a ring homomorphism from $\C H$ to $M_n(\D_G)$. Then under the rational map given by Example \ref{eg. rational}, $\widetilde\D_H$ is a rational $\C^{\times}H$-semiring. Lemma \ref{lem. universal morphism} then gives a map 
\[\Phi_H\colon \rat(\C^{\times}H)\cup\{0\}\rightarrow \widetilde\D_H.\]

\begin{lemma}\label{lem. complexity wd}
The image $\im(\Phi_H)$ equals $\widetilde\D_H$.
\end{lemma}

\begin{proof}
Note that the target of $\Phi_H$ is $\widetilde\D_H$, and so we automatically have the inclusion $\im(\Phi_H)\subset \widetilde\D_H$. To prove the reverse containment, first note that $\im(\Phi_H)$ contains $(\sigma\otimes_{\C}\tau)(\C H)$. Let $x\in \im(\Phi_H)$. Then there exists $\alpha\in \rat(\C^{\times}H)\cup\{0\}$ such that $\Phi_H(\alpha)=x$. If $x$ is invertible in $M_n(\D_G)$, then $x$ is invertible in $\widetilde\D_H$, and then $x^{-1}=x^{\diamond}=\Phi_H(\alpha^{\diamond})\in\im(\Phi_H)$. So $\im(\Phi_H)$ contains the division closure of $(\sigma\otimes_{\C}\tau)(\C H)$ in $M_n(\D_G)$, that is, $\widetilde\D_H$.
\end{proof}

Let $\mathcal{T}$ be the set of finite rooted trees. As in Section \ref{sec. tree}, we get a map
\[\tree\colon \rat(\C^{\times}H)\cup\{0\}\rightarrow \mathcal{T}.\]
The \textit{$H$-complexity} of an element $x\in\widetilde\D_H$ is defined as
\[\tree_H(x)=\min\{\tree(\alpha)\mid \alpha\in \rat(\C^{\times}H)\cup\{0\}, \Phi_H(\alpha)=x\}.\]
By Lemma \ref{lem. complexity wd}, $\tree_H$ is defined on the whole of $\widetilde\D_H$. We say that 
\[\alpha\in \rat(\C^{\times}H)\cup\{0\}\]
\textit{realizes the $H$-complexity} of $x$ if $\Phi_H(\alpha)=x$ and $\tree(\alpha)=\tree_H(x)$.

Now suppose that $H$ is finitely generated and $H=N\rtimes\langle t \rangle$ for some normal subgroup $N\lhd H$ and infinite-order element $t\in H$. For simplicity, we denote $\tau(t) \in \D_H$ by $t$ and $\sigma(t)\cdot \tau(t) \in \widetilde \D_H$ by $s$. Note that conjugation by $t$ induces an automorphism $\D_N\rightarrow \D_N,x\mapsto txt^{-1}$.
Indeed, $\D_N$ is the division closure of $\tau(\C N)$ in $\D_G$, and hence $t \D_N t^{-1}$ is the division closure of $t\tau(\C N)t^{-1} = \tau(t\C Nt^{-1}) = \tau(\C N)$. Therefore, $t \D_N t^{-1}$ and $\D_N$ are division closures of the same ring in $\D_G$, and hence coincide.

We now see that the conjugation $A\mapsto sAs^{-1}$ induces an automorphism of $M_n(\D_N)$. Also, similarly to the above proof, one can show that the conjugation $A\mapsto sAs^{-1}$ induces an automorphism of $\widetilde\D_N$. Therefore, we can form $\D_N(\!(t)\!)$, $\widetilde\D_N(\!(s)\!)$ and $M_n(\D_N)(\!(s)\!)$, the rings of twisted Laurent power series with twisting structures given by these conjugation automorphisms. It is a standard fact that $\D_H$ can be identified with a subring of $\D_N(\!(t)\!)$; it quickly follows for example from  \cite{Kielak2020a}*{Proposition 2.23}.   
It is clear that
$M_n(\D_N(\!(t)\!))=M_n(\D_N)(\!(s)\!)$. Hence, the containment \eqref{eq. subring} implies that $\widetilde\D_H$ and $\widetilde\D_N$ can be identified with subrings of $M_n(\D_N)(\!(s)\!)$. The following is essentially \cite{jaikin2020strong}*{Proposition 5.1}.

\begin{proposition}\label{prop. laurent and complexity}
Let $x\in \widetilde\D_H$ and assume that for every $0\neq y\in\widetilde\D_H$ such that $\tree_H(y)<\tree_H(x)$, $y$ is invertible in $\widetilde\D_H$. Then $x\in \widetilde\D_N(\!(s)\!)$.

Moreover, write $x$ as a Laurent power series
\[x=\sum_i x_is^i,\]
where $x_i\in\widetilde\D_N$ for all $i$. Then
\[\tree_H(x_i)\leqslant \tree_H(x)\]
for all $i$, and equality holds for some $i$ if and only if $x=x_is^i$.
\end{proposition}

\begin{proof}
We will apply \cite{jaikin2020strong}*{Proposition 5.1} with \[\mathcal{A}=M_n(\D_N),\quad \mathcal{P}=M_n(\D_N)(\!(s)\!), \quad \D_{N,\mathcal{P}}=\widetilde\D_N, \quad \D_{H,\mathcal{P}}=\widetilde\D_H\]
in the notation of the proposition. We need to verify that $\widetilde\D_N$ (resp. $\widetilde\D_H$) is the division closure of $(\sigma\otimes_{\C}\tau)(\C N)$ (resp. $(\sigma\otimes_{\C}\tau)(\C H)$) in $M_n(\D_N)(\!(s)\!)$, where we think of $\sigma\otimes_{\C}\tau$ as a ring homomorphism from $\C N$ (resp. $\C H$) to $M_n(\D_N)(\!(s)\!)$.

Consider $\widetilde\D_H$. First, suppose that $y\in\widetilde\D_H$ is invertible in $M_n(\D_N)(\!(s)\!)=M_n(\D_N(\!(t)\!))$. By \eqref{eq. subring},  the entries of the matrix $y$ lie in the skew field $\D_H$. So $y$ is invertible in $M_n(\D_H)$, and thus in $M_n(\D_G)$. So $y$ is invertible in $\widetilde\D_H$. So $\widetilde\D_H$ contains the division closure of $(\sigma\otimes_{\C}\tau)(\C H)$ in $M_n(\D_N)(\!(s)\!)$, say $R$. As a subring of $\widetilde\D_H$, $R$ can be identified with a subring of $M_n(\D_G)$. Let $z\in R$ be such that $z$ is invertible in $M_n(\D_G)$. By \eqref{eq. subring}, $z\in\widetilde\D_H$ is a matrix over $\D_H$. So $z$ is invertible in $M_n(\D_H)$, and thus in $M_n(\D_N(\!(t)\!))$. So $z$ is invertible in $R$. This implies that $R\geqslant \widetilde \D_H$, and therefore finishes the verification that  $\widetilde\D_H$ is the division closure of $(\sigma\otimes_{\C}\tau)(\C H)$ in $M_n(\D_N)(\!(s)\!)$. The verification for $\widetilde\D_N$ is similar and straightforward.

By \cite{jaikin2020strong}*{Proposition 5.1}, we have $x\in\widetilde\D_N(\!(s)\!)$. Write $x$ as a Laurent power series
$x=\sum_i x_is^i$, where $x_i\in\widetilde\D_N$ for all $i$. Now, \cite{jaikin2020strong}*{Proposition 5.1} implies that $\tree_H(x_is^i)\leqslant \tree_H(x)$ for all $i$, and equality holds for some $i$ if and only if $x=x_is^i$. To finish the proof, simply note that for all $z\in\widetilde\D_N$ and all $i$, $\tree_H(z)=\tree_H(zs^i)$.
\end{proof}

\section{Agrarian invariants}\label{sec. agrarian notions}
The notion of agrarian groups was introduced in \cite{kielak2018bieri}, but the idea dates back to Malcev \cite{malcev1948embedding}. The notion of agrarian invariants was later introduced and studied by Henneke and the first author in \cite{henneke2018agrarian}.

\subsection{Agrarian maps}
Let $G$ be a group. An \textit{agrarian map} of $G$ is a finite dimensional left linear representation $\sigma\colon G\rightarrow \GL_n(D)$ of $G$ over a skew field $D$.

\begin{remark}
In \cite{henneke2018agrarian}, the authors define an agrarian map to be a $1$-dimensional representation $G\rightarrow \GL_1(D)$ over a skew field. As we will see in Examples \ref{eg. l2 betti} and \ref{eg. alexander norm}, general finite-dimensional representations arise naturally in the study of twisted invariants. We therefore generalize the work of \cite{henneke2018agrarian} and define an agrarian map to be a general finite-dimensional representation.
\end{remark}

\subsection{Rationalization}\label{sec. rationalization}
If $G$ is finitely generated, then it has a maximal free abelian quotient denoted $G_{\mathrm{fab}}$. Let $q\colon G\twoheadrightarrow G_{\mathrm{fab}}$ be the natural quotient map, and let $D G_{\mathrm{fab}}$ and $\Q G_{\mathrm{fab}}$ be the untwisted group rings. View $q\colon G\rightarrow \ore(\Q G_{\mathrm{fab}})$ as another representation and form the tensor product representation
\[\sigma\otimes_{\Z} q\colon G\rightarrow \GL_n(\ore(DG_{\mathrm{fab}})),~ (\sigma\otimes_{\Z} q)(g)=\sigma(g)q(g),\]
called the \textit{rationalization} of $\sigma$.

For the rest of this subsection suppose $n=1$, i.e., we have a homomorphism $\sigma \colon G\rightarrow D^{\times}$. In \cites{henneke2018agrarian}, Henneke and the first author introduce the following rationalization of the agrarian map $\sigma$. Let $K=\ker q$. As explained in Example \ref{eg. twisted group ring}, by picking a set-theoretic section $s\colon G_{\mathrm{fab}}\rightarrow G$, we obtain a twisted group ring $(\Z K)G_{\mathrm{fab}}$ with a natural isomorphism $\Z G\cong (\Z K) G_{\mathrm{fab}}$. The maps $s$ and $\sigma$ together induce a twisted group ring structure $D \ast G_{\mathrm{fab}}$ (here the notation is used to distinguish the twisted group ring $D \ast G_{\mathrm{fab}}$ from the untwisted group ring $D G_{\mathrm{fab}}$). The restriction $\sigma|_{\Z K}:\Z K\rightarrow D$ naturally induces a ring homomorphism $(\Z K) G_{\mathrm{fab}}\rightarrow D \ast G_{\mathrm{fab}}$ between the twisted group rings. Note that there is a natural embedding $D\ast G_{\mathrm{fab}}\hookrightarrow \ore(D\ast G_{\mathrm{fab}})$. By definition, the \textit{HK-rationalization} of $\sigma$, denoted $\ts\colon \Z G\rightarrow \ore(D \ast G_{\mathrm{fab}})$, is the composition \[\Z G\cong (\Z K) G_{\mathrm{fab}}\rightarrow D\ast G_{\mathrm{fab}}\rightarrow \ore(D\ast G_{\mathrm{fab}}).\]

The following Lemma \ref{lem. equivalent rationalization} can be easily extracted from the proof of \cite{jaikin2020universality}*{Proposition 3.5}. We provide the proof for the convenience of the reader. Roughly speaking,  \cref{lem. equivalent rationalization} implies that $\sigma\otimes_{\Z}q$ and $\ts$ are equivalent for the purpose of defining agrarian invariants. The reader is referred to Remark \ref{rm. equivalent rationalization} below for the precise meaning of this equivalence.

\begin{lemma}\label{lem. equivalent rationalization}
There is an isomorphism $\alpha\colon \ore(D\ast G_{\mathrm{fab}}) \rightarrow \ore(DG_{\mathrm{fab}})$ such that the following hold.
\begin{enumerate}[label=(\roman*)]
    \item\label{item. equivalent rationalization 1} $\alpha\circ\ts=\sigma\otimes_{\Z}q$.
    \item\label{item. equivalent rationalization 2}
    For all $x\in D   G_{\mathrm{fab}}$, view $x$ and $\alpha(x)$ as functions $x,\alpha(x)\colon G_{\mathrm{fab}}\rightarrow D$. Then $\supp(x)=\supp(\alpha(x))$.
\end{enumerate}

\end{lemma}

\begin{proof}
For all $d\in D$ and $h\in G_{\mathrm{fab}}$, let
\[\alpha(d\ast h)=d\cdot \sigma(s(h))\cdot h,\]
where the product $d\ast h$ on the left-hand side is considered as an element of $D \ast G_{\mathrm{fab}}$, and the right-hand side product is considered as an element of $DG_{\mathrm{fab}}$. By extending $\alpha$ linearly across $D \ast G_{\mathrm{fab}}$ we get a map $\alpha\colon D \ast G_{\mathrm{fab}} \rightarrow DG_{\mathrm{fab}}$. Item \ref{item. equivalent rationalization 2} follows immediately.

We check that $\alpha$ is a ring homomorphism. First, $\alpha$ obviously preserves addition. Second, note that the identity of $D\ast G_{\mathrm{fab}}$ is $\sigma(s(1))^{-1}\ast 1$. Indeed, for all $d\in D$ and $h\in G_{\mathrm{fab}}$,
\begin{align*}
    \Big(\sigma\big(s(1)\big)^{-1}\ast 1\Big)\cdot (d\ast h)
    &=\Big(\sigma\big(s(1)\big)^{-1}\sigma\big(s(1)\big)d\sigma\big(s(1)\big)^{-1}\sigma\big(s(1)\big)\sigma\big(s(h)\big)\sigma\big(s(h)\big)^{-1}\Big)\ast h\\
    &=d\ast h.
\end{align*}
Since $\alpha(\sigma(s(1))^{-1}\ast 1)=\sigma(s(1))^{-1}\cdot \sigma(s(1))\cdot 1=1$, the function $\alpha$ preserves the identity.

Third, for all $d_1,d_2\in D$ and $h_1,h_2\in G_{\mathrm{fab}}$,
\begin{align*}
    \alpha\big((d_1\ast h_1)\cdot(d_2\ast h_2)\big)
    &=\alpha\Big(\Big(d_1\sigma\big(s(h_1)\big)d_2\sigma\big(s(h_1)\big)^{-1}\sigma\big(s(h_1)\big)\sigma\big(s(h_2)\big)\sigma\big(s(h_1h_2)\big)^{-1}\Big)\ast (h_1h_2)\Big)\\
    &=d_1\sigma\big(s(h_1)\big)d_2\sigma\big(s(h_2)\big)\cdot (h_1h_2), \\
    \\
\alpha(d_1\ast h_1)\cdot\alpha(d_2\ast h_2) &=d_1\sigma\big(s(h_1)\big)h_1\cdot d_2\sigma\big(s(h_2)\big)h_2 \\
&=d_1\sigma\big(s(h_1)\big)d_2\sigma\big(s(h_2)\big)\cdot (h_1h_2).
\end{align*}
Thus, $\alpha$ preserves multiplication. This show that $\alpha$ is a ring homomorphism.

We define an inverse of $\alpha$ as follows. Let $\beta\colon DG_{\mathrm{fab}}\rightarrow D \ast G_{\mathrm{fab}}$ be given by setting
\[\beta(dh)=\Big(d\sigma\big(s(h)\big)^{-1}\Big)\ast h,\]
for all $d\in D,h\in G_{\mathrm{fab}}$,
and then extending  linearly across $DG_{\mathrm{fab}}$. It is easy to check that $\beta$ is indeed the inverse of $\alpha$, and thus $\alpha\colon D\ast G_{\mathrm{fab}} \rightarrow DG_{\mathrm{fab}}$ is a ring isomorphism.

Therefore, $\alpha$ extends to a ring isomorphism between the Ore localizations. It remains to check item \ref{item. equivalent rationalization 1}. Let $g\in G$, let $h=q(g)\in G_{\mathrm{fab}}$. Then
\[\alpha(\ts(g))=\alpha(\sigma(gs(h)^{-1})\ast h)=\sigma(gs(h)^{-1})\sigma(s(h))\cdot h=\sigma(g)\cdot h=\sigma(g)q(g).\]
as desired.
\end{proof}

\subsection{Agrarian Betti numbers}
In the sequel, tensor products happen between left and right modules, i.e., if $R$ is a ring, then $M\otimes_R N$ is the tensor product of a right $R$-module $M$ and a left $R$-module $N$. Moreover, if $S\subset M,T\subset N$ are not submodules, then we define
\[S\otimes_R T=\{s\otimes t\mid s\in S, t\in T\}.\]

Let $(C_{\ast},\partial_{\ast})$ be a chain complex of free right $\Z G$-modules. Consider the left $M_n(D)$-module $D^n$ of column vectors, and note that $\sigma$ endows $D^n$ with the structure of a left $G$-module with action given by $g\cdot x=\sigma(g)(x)$.

Tensoring $C_{\ast}$ with $D^n$ over $\Z G$ gives rise to a chain complex $C_{\ast}\otimes_{\Z G} D^n$, which is a complex of right $D$-modules. So $H_{\ast}(C_{\ast}\otimes_{\Z G} D^n)$ is also a right $D$-module. The $i^{th}$ $\sigma$\textit{-agrarian Betti number} of $C_{\ast}$ is
\[b^{\sigma}_i(C_{\ast})=\dim_D H_i(C_{\ast}\otimes_{\Z G} D^n).\]

If $C_{\ast}$ is a projective resolution of $\Z$ over $\Z G$, then we obtain the $i^{th}$ $\sigma$\textit{-agrarian Betti number} of $G$:
\[b^{\sigma}_i(G)=\dim_D H_i(G, D^n)\]
where the homology is computed with respect to the representation $\sigma$. 

\begin{example}[Twisted $\ell^2$-Betti numbers]\label{eg. l2 betti}
Suppose that $G$ is locally indicable (twisted $\ell^2$-Betti numbers can be defined for any group, but for the purpose of this paper we restrict ourselves to locally indicable ones). Let $\eta\colon G\rightarrow \GL_n(\C)$ be a finite-dimensional complex representation. Recall that $\D_G$ denotes the Linnell skew field of $G$ and $\tau\colon \Z G\rightarrow \D_G$ denotes the natural embedding. Let
\[\sigma=\eta\otimes_{\C}\tau \colon G\rightarrow \GL_n(\D_G),~\sigma(g)=\eta(g)\tau(g).\]
Then the $i^{th}$ \textit{twisted} $\ell^2$\textit{-Betti number} of $C_{\ast}$ with respect to $\eta$ is
\[b^{(2),\eta}_{i}(C_{\ast})=b^{\sigma}_i(C_{\ast}).\]
If $\eta$ is the trivial representation $G\rightarrow \GL_1(\C)$ that sends every $g\in G$ to $1$, then we denote $b^{(2),\eta}_{i}(C_{\ast})$ by $b^{(2)}_i(C_{\ast})$ and call it the $i^{th}$ \textit{(untwisted)} $\ell^2$\textit{-Betti number} of $C_{\ast}$.
\end{example}

\begin{lemma}
Suppose that $G$ is finitely generated with $G_{\mathrm{fab}}$ non-trivial. Then $b^{\sigma\otimes_{\Z} q}_0(G)=0$.
\end{lemma}

\begin{proof}
There exists a presentation $G=\langle X \mid \mathcal{R} \rangle$ with an element $x\in X$ such that $q(x)$ is an element of a basis $Y$ of $G_{\mathrm{fab}}$. Below, we denote $\ore(DG_{\mathrm{fab}})$ by $D(Y)$ and when we talk about the $q(x)$-order, we mean the order with respect to $Y$. Note that the action of $G$ on $(D(Y))^n$ is given by the following: elements $v\in (D(Y))^n$ are column vectors, and we have $g\cdot v= \sigma(g)\cdot q(g)\cdot v$, i.e., the vector obtained by multiplying $v$ by the matrix $\sigma(g)\cdot q(g)$ on the left.

Construct a $K(G,1)$ CW-complex $BG$ from the presentation complex of $G=\langle X \mid \mathcal{R} \rangle$ by adding cells of dimension greater than or equal to $3$, and let $EG$ be the universal cover of $BG$. Consider the cellular chain complex of $EG$ consisting of right $G$-modules:
\[\cdots\rightarrow C_1\xrightarrow[]{\partial_1}C_0\rightarrow 0.\]

Let $p$ be the unique $0$-dimensional cell of $BG$ and let $e$ be the edge of $BG$ labeled by $x$. There exists a lift $\widetilde{p}$ (resp. $\widetilde{e}$) of $p$ (resp. $e$) such that
\[\partial_1(\widetilde{e})=\widetilde p\cdot (1-x).\]

Consider the map
\[\partial_1\otimes_{\Z G} \mathrm{id}_{(D(Y))^n}\colon C_1\otimes_{\Z G} ((D(Y))^n\rightarrow C_0\otimes_{\Z G} (D(Y))^n.\]
Let 
\[\partial'_1\colon  (\widetilde{e}\cdot(\Z G))\otimes_{\Z G} ((D(Y))^n\rightarrow C_0\otimes_{\Z G} ((D(Y))^n\]
be the restriction of $\partial_1\otimes_{\Z G} \mathrm{id}_{((D(Y))^n}$. Note that for every column vector $v\in (D(Y))^n$, we have
\[\partial'_1(\widetilde e\otimes v)=(\widetilde p \cdot(1-x))\otimes v=\widetilde p \otimes (v-\sigma(g)\cdot q(g)\cdot v).\]
Let $\mathcal{B}$ be the standard $D(Y)$-basis of $((D(Y))^n$. Then under the bases $\{\widetilde{e}\}\otimes_{\Z G} \mathcal{B}$ and $\{\widetilde{p}\}\otimes_{\Z G} \mathcal{B}$, the matrix representative of $\partial'_1$ has the form
\[\mathrm{Id}-\sigma(x)\cdot q(x).\]
Since $\sigma(x)$ is a matrix over $D$, each entry of $\sigma(x)q(x)$ has $q(x)$-order at least $1$. So \cref{lem. matrix reduction} implies that $\mathrm{Id}-\sigma(x)\cdot q(x)$ is invertible, and thus $\partial'_1$ is surjective. It follows that $\partial_1\otimes_{\Z G} \mathrm{id}_{(D(Y))^n}$ is surjective as well and thus $b^{\sigma\otimes_{\Z} q}_0(G)=0$.
 \end{proof}

The $\sigma$\textit{-agrarian Euler characteristic} of $C_{\ast}$ is
\[\chi^{\sigma}(C_{\ast})=\sum^{\infty}_{i=0}(-1)^i b^{\sigma}_i(C_{\ast})\]
provided that the sum is well-defined, i.e., only finitely many of the values $b^{\sigma}_i(C_{\ast})$ are non-zero, and all non-zero terms are finite.

If in addition $C_{\ast}$ is a projective resolution of $\Z$ over $\Z G$, then we obtain the $\sigma$\textit{-agrarian Euler characteristic} of $G$:
\[\chi^{\sigma}(G)=\sum^{\infty}_{i=0}(-1)^i b^{\sigma}_i(G).\]

\begin{proposition}\label{prop. euler characteristic}
If $C_{\ast}$ is \emph{finite}, i.e., each $C_i$ has finite rank and there are only finitely many non-zero modules $C_i$, then 
\[\sum^{\infty}_{i=0}(-1)^i b^{\sigma}_i(C_{\ast})=n\cdot \sum^{\infty}_{i=0}(-1)^i\rank_{\Z G}C_i=n\cdot \chi(C_{\ast}),\]
where $\chi(C_{\ast})$ denotes the usual Euler characteristic of $C_{\ast}$. In particular, 
\[\chi^{\sigma}(G)=n\cdot\chi(G)\]
if $G$ is of type $F$.
\end{proposition}

\begin{proof}
For all $i$, decompose $C_i\otimes_{\Z G} D^n$ as 
\[C_i\otimes_{\Z G} D^n=B_i\oplus H_i\oplus C'_i,\]
where $B_i$ is the image of $\partial_i\otimes_{\Z G} \mathrm{id}_{D^n}$, $H_i\cong H_i(C_{\ast}\otimes_{\Z G} D^n)$, and $\partial_i\otimes_{\Z G} \mathrm{id}_{D^n}$ maps $C'_i$ isomorphically onto $B_{i-1}$ and restricts to the trivial map on $B_i \oplus H_i$. Then
\[\dim_D B_i=\dim_D C'_{i+1}.\]
Therefore,
\[\sum^{\infty}_{i=0}(-1)^i \dim_D H_i=\sum^{\infty}_{i=0} (-1)^i \dim_D (C_i\otimes_{\Z G} D^n)=n\cdot \sum^{\infty}_{i=0} (-1)^i  \rank_{\Z G}(C_i).\qedhere\]
\end{proof}

\begin{proposition}\label{prop. mapping torus and vanishing betti number}
Suppose that $G$ is finitely generated and is the fundamental group of the mapping torus $T_f$ of a cellular self-map $f\colon Y\rightarrow Y$ of a  connected CW-complex $Y$ with finite $d$-skeleton. Let $C_{\ast}$ be the cellular chain complex of the universal cover $\widetilde{T_f}$ of $T_f$. Then for $i\leqslant d$,
\[b^{\sigma\otimes_{\Z} q}_i(C_{\ast})=0.\]
\end{proposition}

\begin{proof}
For each $i$, by lifting each $i$-cell of $Y\subset T_f$ to an $i$-cell in the universal cover $\widetilde{T_f}$, we obtain a set $\B_i\subset C_i$. Let
\[\A_{i+1}=\{\Delta\times[0,1]\mid \Delta\in \B_i\},\]
\[A_{i+1}=\Span_{\Z G}\A_i,~~B_i=\Span_{\Z G}\B_i.\]
(Here the subscript keeps track of dimension and so we use $\A_{i+1}$ instead of $\A_i$.) Then $\A_i\cup\B_i$ is a $\Z G$-basis of $C_i=A_i\oplus B_i$. Let $V$ be the standard basis of $(\ore(D G_\mathrm{fab}))^n$. Then $\A_i\otimes_{\Z G} V$ (resp. $\B_i\otimes_{\Z G} V$) is an $\ore(D G_\mathrm{fab})$-basis of $A_i\otimes_{\Z G} (\ore(D G_\mathrm{fab}))^n$ (resp. $B_i\otimes_{\Z G} (\ore(D G_\mathrm{fab}))^n$).

Now suppose $i\leqslant d$. Let 
\[P_i\colon C_i\otimes_{\Z G} (\ore(D G_\mathrm{fab}))^n\rightarrow B_i\otimes_{\Z G} (\ore(D G_\mathrm{fab}))^n\]
be the projection corresponding to the direct sum decomposition 
\[C_i\otimes_{\Z G} (\ore(D G_\mathrm{fab}))^n=(A_i\otimes_{\Z G} (\ore(D G_\mathrm{fab}))^n)\oplus (B_i\otimes_{\Z G} (\ore(D G_\mathrm{fab}))^n),\]
and let $\partial'_{i+1}$ be the restriction of $\partial_{i+1}\otimes_{\Z G} \mathrm{id}_{(\ore(D G_\mathrm{fab}))^n}$ to $A_{i+1}$. 

The natural map $T_f\rightarrow S^1$ that maps $Y$ to a single point induces a group homomorphism $G\rightarrow \Z$, which factors through a homomorphism $G_{\mathrm{fab}}\rightarrow \Z$. Let $X$ be a basis of the free abelian group $G_{\mathrm{fab}}$ such that there exists $t\in X$ that is mapped to a generator of $\Z$ by the homomorphism $G_{\mathrm{fab}}\rightarrow \Z$. Below, we denote $\ore(DG_{\mathrm{fab}})$ by $D(X)$ and when we talk about $t$-order, we mean the order with respect to the basis $X$.

The matrix representative of $P_i\circ \partial'_{i+1}$ under the bases $\A_{i+1}\otimes_{\Z G} V$ and $\B_i\otimes_{\Z G} V$ has the form $\mathrm{Id}+M\cdot t$ where $M$ is a matrix over $D[X\smallsetminus\{t\}]$. In particular, each entry of $M\cdot t$ has $t$-order at least $1$. By Lemma \ref{lem. matrix reduction}, $P_i\otimes_{\Z G} \partial'_{p+1}$ is invertible. In particular, 
\begin{align*}
 \dim_{D(X)}\im (\partial_{i+1}\otimes_{\Z G} \mathrm{id}_{(D(X))^n})
 & \geqslant \dim_{D(X)}\im (\partial'_{i+1}\otimes_{\Z G} \mathrm{id}_{(D(X))^n}) \\
 &= \dim_{D(X)} (B_i\otimes_{\Z G} (D(X))^n) \\
 &= n\cdot |\B_i|.
\end{align*}

The same argument with $i-1$ in place of $i$ shows that
\begin{align*}
    & \hspace{-10mm} \dim_{D(X)} \ker(\partial_i\otimes_{\Z G} \mathrm{id}_{(D(X))^n})\\
    &= \dim_{D(X)}C_i\otimes_{\Z G} (D(X))^n-\dim_{D(X)} \im(\partial_i\otimes_{\Z G} \mathrm{id}_{(D(X))^n})\\
    &\leqslant  n\cdot(|\A_i|+|\B_i|)-n\cdot |\B_{i-1}|\\
    &= n\cdot |\B_i|,
\end{align*}
where the last equality follows from $|\A_i|=|\B_{i-1}|$.

So
\begin{align*}
    b^{\sigma\otimes_{\Z} q}_i(C_{\ast})     &=\dim_{D(X)} \ker(\partial_i\otimes_{\Z G} \mathrm{id}_{(D(X))^n})-\dim_{D(X)}\im (\partial_{i+1}\otimes_{\Z G} \mathrm{id}_{(D(X))^n})\\
    &\leqslant 0.
\end{align*}

As the reverse inequality automatically holds, the desired result follows.
\end{proof}

\begin{corollary}\label{cor. fxz}
Suppose that $G$ is a (type $F$)-by-(infinite cyclic) group. Then $b^{\sigma\otimes_{\Z} q}_{\ast}(G)=0$.
\end{corollary}

\subsection{Agrarian torsion}
Suppose that $C_{\ast}$ is finite, that it comes with a preferred basis $\B_{C_{\ast}}$, and that $C_{\ast}$ is $\sigma$\textit{-acyclic}, i.e., $b^{\sigma}_i(C_{\ast})=0$ for all $i$, which implies that $C_\ast \otimes_{\Z G} D^n$ is contractible (see, e.g., \cite{rosenberg1994algebraic}*{Proposition 1.7.4}). Let $\gamma$ (resp. $d$) be a chain contraction (resp. the boundary map) of $C_{\ast}\otimes_{\Z G} D^n$. Then $d+\gamma \colon  C_{\mathrm{even}}\otimes_{\Z G} D^n \rightarrow C_{\mathrm{odd}}\otimes_{\Z G} D^n$ is an isomorphism of right $D$-modules, where $C_{\mathrm{even}}$ (resp. $C_{\mathrm{odd}}$) is the direct sum of the even (resp. odd) dimensional components of $C_{\ast}$. Tensoring the preferred basis of $C_{\ast}$ with the standard basis $V$ of $D^n$ gives rise to a preferred basis $\B_{\ast}=\B_{C_{\ast}}\otimes_{\Z G}V$ for $C_{\ast}\otimes_{\Z G} D^n$. Represent $d+\gamma$ by a matrix $M$ over $D$ using $\B_{\ast}$. The $\sigma$\textit{-agrarian torsion}, denoted by $\rho_{\sigma}(C_{\ast})$, is the Dieudonn\'{e} determinant $\Det_D M$. The value of $\rho_{\sigma}(C_{\ast})$ does not depend on the choice of $\gamma$ (see, e.g., \cite{cohen1973course}*{(15.3)}).

For computational purposes we record the following:

\begin{remark}\label{rm. change of basis}
Suppose that $\B'_{\ast}$ is another basis of $C_{\ast}\otimes_{\Z G} D^n$ and the change of basis matrix from $\B_{\ast}$ to $\B'_{\ast}$ has determinant $\pm 1$. Let $N$ be the matrix representative over $d+\gamma$ under the new basis $\B'_{\ast}$. Then $\Det_D N=\pm\Det_D M$.
\end{remark}

\begin{remark}\label{rm. computation}
Let $f\colon  A\rightarrow B$ be a homomorphism between based finite-rank free right $\Z G$-modules and suppose the matrix representative of $f$ under the chosen bases is $M$ (so $f$ coincides with left-multiplication by $M$). Tensoring these bases with the standard basis for $D^n$ yields bases for $A\otimes_{\Z G} D^n$ and $B\otimes_{\Z G} D^n$. The matrix representative of $f\otimes_{\Z G} \mathrm{id}_D$ under these bases is given by $\sigma(M)$ (since $A\otimes_{\Z G} D^n$ and $B\otimes_{\Z G} D^n$ are right $D$-modules, $f\otimes_{\Z G} \mathrm{id}_D$ is given by left-multiplication by $\sigma(M)$). Here, $\sigma(M)$ is the matrix obtained by applying $\sigma$ to each entry of $M$.
\end{remark}

\subsection{Agrarian polytope and agrarian norm}
Now further suppose that $G$ is finitely generated. Recall that there is a natural map $q \colon  G\twoheadrightarrow G_{\mathrm{fab}}$ of $G$ onto its maximal free abelianization $G_{\mathrm{fab}}$. Let $\sigma\colon G\rightarrow \GL_n(D)$ be a representation over a skew field $D$. Suppose that $C_{\ast}$ is $(\sigma\otimes_{\Z} q)$-acyclic. Then the $\sigma$\textit{-agrarian polytope} with respect to $C_{\ast}$ is $P(\rho_{\sigma\otimes_{\Z} q}(C_{\ast}))$, where
\[P\colon (\ore(DG_{\mathrm{fab}}))^{\times}_{\mathrm{ab}}\rightarrow \p(G_{\mathrm{fab}})\]
is the polytope homomorphism defined in Section \ref{sec. rational function}. Let $\phi\in H^1(G,\Z)$ be a character. The $\sigma$\textit{-agrarian norm} of $\phi$ is defined as
\[\|\phi\|_{\sigma,C_{\ast}}=\sup\big\{\phi(z)\mid z\in P(\rho_{\sigma\otimes_{\Z} q}(C_{\ast}))\big\}-\inf\big\{\phi(z)\mid z\in P(\rho_{\sigma\otimes_{\Z} q}(C_{\ast}))\big\}.\]
We will prove in Section \ref{sec. aspherical group} that the agrarian norm is a semi-norm in many interesting cases, justifying the terminology.

A second way of defining $\|\cdot\|_{\sigma,C_{\ast}}$, which is computationally easier, is the following. First find a character $\psi\in H^1(G,\Z)$ such that $\phi=k\psi$ for some $k\in\mathbb{N}$ and $\psi$ is a primitive integral character. Choose a basis $X$ of $G_{\mathrm{fab}}$ such that there is an $x\in X$ with $\psi(x)=1$ and $\psi(y)=0$ for all $y\in X\smallsetminus\{x\}$. Define
\[\|\phi\|_{\sigma,C_{\ast}}=k\cdot\deg_x \rho_{\sigma\otimes_{\Z} q}(C_{\ast}),\]
where $\deg_x$ is computed with respect to the basis $X$.

\begin{remark}\label{rm. equivalent rationalization}
In the case $n=1$, \cite{henneke2018agrarian} provides an alternative definition for the agrarian polytope, which we call the \textit{HK-polytope} for $\sigma$ for the moment. The HK-polytope for $\sigma$ is defined using $\ts\colon G \rightarrow \ore(D \ast G_{\mathrm{fab}})$, the HK-rationalization of $\sigma$, where $D \ast G_{\mathrm{fab}}$ denotes the twisted group ring in Section \ref{sec. rationalization}. We sketch the construction here and refer the reader to \cite{henneke2018agrarian} for details. The HK-polytope is defined only when $b^{\ts}_i(C_{\ast})=0$ for all $i$, so let us assume this is indeed the case. Generalizing Section \ref{sec. rational function}, one can define a polytope homomorphism $\widetilde P\colon \ore(D \ast G_{\mathrm{fab}})^{\times}_{\mathrm{ab}}\rightarrow \mathcal{P}(G_{\mathrm{fab}})$. Then the HK-polytope for $\sigma$ is $\widetilde P(\rho_{\ts}(C_{\ast}))$, where $\rho_{\ts}$ denotes the $\ts$-agrarian torsion.

The isomorphism $\alpha\colon \ore(D \ast G_{\mathrm{fab}})\rightarrow \ore(DG_{\mathrm{fab}})$ provided by Lemma \ref{lem. equivalent rationalization} implies that $b^{\ts}_i(C_{\ast})=0$ if and only if $b^{\sigma\otimes_{\Z}q}_i(C_{\ast})=0$. So the HK-polytope for $\sigma$ is well-defined if and only if our $\sigma$-agrarian polytope is well-defined. Moreover, let $x\in D \ast G_{\mathrm{fab}}$ and consider $\alpha(x)$. View $x$ and $\alpha(x)$ as functions $x,\alpha(x)\colon G_{\mathrm{fab}}\rightarrow D$. Then $\supp(x)=\supp(\alpha(x))$. It follows that HK-polytope coincides with our agrarian polytope. The benefit of our approach is that it only uses the untwisted group ring, which is computationally simpler.
\end{remark}

Before giving examples of the agrarian norm we would like to first prove its homotopy invariance. The proof of the following proposition combines ideas of \cite{henneke2018agrarian} and \cite{kielak2018bieri}.

\begin{proposition}[Homotopy invariance]\label{prop. homotopy invariance}
Let $C_{\ast},C'_{\ast}$ be homotopy equivalent finite based chain complexes of free  $\Z G$-modules. Suppose that $C_{\ast}$ is $(\sigma\otimes_{\Z} q)$-acyclic. Then so is $C'_{\ast}$ and there is an equality between the agrarian polytopes
\[P(\rho_{\sigma\otimes_{\Z}q}(C_{\ast}))=P(\rho_{\sigma\otimes_{\Z}q}(C'_{\ast})).\]
In particular, there is an equality between the corresponding agrarian norms
\[\|\cdot\|_{\sigma,C_{\ast}}=\|\cdot\|_{\sigma,C'_{\ast}}.\]
\end{proposition}

\begin{proof}
Let $f\colon  C_{\ast}\rightarrow C'_{\ast}$ be a (chain) homotopy equivalence. Then $f\otimes_{\Z G} \mathrm{id}_{(\ore(DG_{\mathrm{fab}}))^n}$ is a homotopy equivalence between $C_{\ast}\otimes_{\Z G} (\ore(DG_{\mathrm{fab}}))^n$ and $C'_{\ast}\otimes_{\Z G} (\ore(DG_{\mathrm{fab}}))^n$. Thus, ${C'_{\ast}\otimes_{\Z G} (\ore(DG_{\mathrm{fab}}))^n}$ is acyclic. 

Consider the mapping cone $\cone_{\ast}(f)$ with basis the union of bases of $C_{\ast}$ and $C'_{\ast}$. Since $f$ is a homotopy equivalence, $\cone_{\ast}(f)$ is contractible and hence its Whitehead torsion $\rho(\cone_{\ast}(f))$ is defined. Moreover, $\cone_{\ast}(f)\otimes_{\Z G} (\ore(DG_{\mathrm{fab}}))^n$ is also contractible and has (see \cref{rm. computation}) 
\[\rho_{\sigma\otimes_{\Z}q}(\cone_{\ast}(f))=\Det_{\ore(DG_{\mathrm{fab}})}\left((\sigma\otimes_{\Z}q)\left(\rho\left(\cone_{\ast}(f)\right)\right)^{-1}\right).\]

There is a short exact sequence
\[0\rightarrow C'_{\ast}\rightarrow \cone_{\ast}(f) \rightarrow \Sigma C_{\ast} \rightarrow 0,\]
where $\Sigma C_{\ast}$ is the suspension of $C_{\ast}$. Since \[\cone_{\ast}(f\otimes_{\Z G} \mathrm{id}_{(\ore(DG_{\mathrm{fab}}))^n})=\cone_{\ast}(f)\otimes_{\Z G} (\ore(DG_{\mathrm{fab}}))^n\]
and 
\[\Sigma(C_{\ast}\otimes_{\Z G} (\ore(DG_{\mathrm{fab}}))^n)=\Sigma C_{\ast} \otimes_{\Z G} (\ore(DG_{\mathrm{fab}}))^n,\]
the above short exact sequence is still exact after tensoring with $(\ore(DG_{\mathrm{fab}}))^n$. Now, \cite{cohen1973course}*{(17.2)} (thinking of these modules as $(D,\{\pm 1\})$-modules in the sense of \cite{cohen1973course}) yields
\[\rho_{\sigma\otimes_{\Z}q}(C'_{\ast})\cdot (\rho_{\sigma\otimes_{\Z}q}(C_{\ast}))^{-1}=\rho_{\sigma\otimes_{\Z}q}(\cone_{\ast}(f))=\Det_{\ore(DG_{\mathrm{fab}})}\left((\sigma\otimes_{\Z}q)\left(\rho\left(\cone_{\ast}(f)\right)\right)^{-1}\right),\]
and thus 
\begin{equation*}
    P(\rho_{\sigma\otimes_{\Z}q}(C'_{\ast}))-P(\rho_{\sigma\otimes_{\Z}q}(C_{\ast}))=P\left(\Det_{\ore(DG_{\mathrm{fab}})}\left((\sigma\otimes_{\Z}q)\left(\rho\left(\cone_{\ast}(f)\right)\right)^{-1}\right)\right).
\end{equation*}
Since $(\sigma\otimes_{\Z}q)(\rho(\cone_{\ast}(f)))$ is a matrix over $DG_{\mathrm{fab}}$, $P((\sigma\otimes_{\Z}q)(\rho(\cone_{\ast}(f))))$ is a single polytope by \cref{thm. single polytope}. Since $\rho(\cone_{\ast}(f))$ is invertible over $\Z G$, ${(\sigma\otimes_{\Z}q)((\rho(\cone_{\ast}(f)))^{-1})}$ is well-defined and is a matrix over $DG_{\mathrm{fab}}$. \cref{thm. single polytope} then implies that 
\[P\left((\sigma\otimes_{\Z}q)\left(\left(\rho\left(\cone_{\ast}(f)\right)\right)^{-1}\right)\right)\]
is also a single polytope. We have
\[P((\sigma\otimes_{\Z}q)(\rho(\cone_{\ast}(f))))+P\left((\sigma\otimes_{\Z}q)\left(\left(\rho\left(\cone_{\ast}(f)\right)\right)^{-1}\right)\right)=P(\mathrm{Id})=0,\]
and so $P((\sigma\otimes_{\Z}q)(\rho(\cone_{\ast}(f))))=-P((\sigma\otimes_{\Z}q)((\rho(\cone_{\ast}(f)))^{-1}))$. But since both are single polytopes, $P((\sigma\otimes_{\Z}q)(\rho(\cone_{\ast}(f))))=P((\sigma\otimes_{\Z}q)((\rho(\cone_{\ast}(f)))^{-1}))=0$.
\end{proof}

Suppose that $G$ is of type $F$ and is $(\sigma\otimes_{\Z} q)$-acyclic. Let $C_{\ast},C'_{\ast}$ be two finite type based free resolutions of $\Z$ over $\Z G$. The above proposition then implies that $\|\cdot\|_{\sigma,C_{\ast}}=\|\cdot\|_{\sigma,C'_{\ast}}$ and thus the agrarian norm does not depend on the choice of resolution. In this case, we will simply denote $\|\cdot\|_{\sigma,C_{\ast}}$ by $\|\cdot\|_{\sigma}$ and call it the $\sigma$\textit{-agrarian norm} of $G$.

\begin{example}[Thurston norm]\label{eg. thurston norm}
Suppose that $G$ is locally indicable and (type $F$)-by-(infinite cyclic). Let $\tau\colon  \C G\hookrightarrow \D_G$ be the embedding of $\C G$ into the Linnell skew field. Then the \textit{Thurston norm}, denoted by $\|\cdot\|_T$, is the $\tau$-agrarian norm $\|\cdot\|_{\tau}$. If $G$ is the fundamental group of a closed $3$-manifold $M$ that fibers over the circle $S^1$, then $\|\cdot\|_T$ is exactly the classical Thurston semi-norm of $M$ \cites{friedl2017universal,liu2017degree}. Let $P(\rho_{\tau\otimes_{\Z}q})$ be the corresponding agrarian polytope. Then $2P(\rho_{\tau\otimes_{\Z}q})$ is the dual Thurston polytope \cite{friedl2017universal}*{Theorem 3.35}.
\end{example}

\begin{example}[Twisted Alexander norm]\label{eg. alexander norm}
Suppose that $G$ is (type $F$)-by-(infinite cyclic) and $\sigma\colon G\rightarrow \GL_n(\C)$ is a complex representation. The \textit{twisted Alexander norm with respect to $\sigma$} is the agrarian norm $\|\cdot\|_{\sigma}$. If $\sigma$ is the trivial representation $G\rightarrow \GL_1(\C)$ that sends every $g\in G$ to $1$, then $\|\cdot\|_{\sigma}$ is called the \textit{(untwisted) Alexander norm}.
\end{example}

\section{Twisted \texorpdfstring{$\ell^2$}{l2}-Betti numbers}\label{sec. twisted l2 betti}

This section is devoted to the proof of the following result.

\begin{theorem}\label{thm. l2 betti fibration}
Let $F\rightarrow E \rightarrow B$ be a fibration of connected finite CW-complexes, or more generally, topological spaces that are homotopy equivalent to connected finite CW-complexes. Suppose that $\pi_1(B)$ is virtually locally indicable. If $F$ is simply connected, or more generally, if the map $\pi_1(E)\rightarrow\pi_1(B)$ induced by the fibration is an isomorphism, then:
\begin{enumerate}[label=(\roman*)]
    \item\label{item. l2 betti 1} For all $i \in \mathbb N$ we have $b^{(2)}_i(E)\leqslant \sum^i_{j=0}b_j(F)\cdot b^{(2)}_{i-j}(B)$.
    \item\label{item. l2 betti 2} If the homology of $F$ with $\C$-coefficients is non-zero in at most two degrees, $0$ and $n$ with $n\geqslant \max \{2, \dim B \}$, then for every $i \in \mathbb N$ we have
    \[b^{(2)}_i(E)=b^{(2)}_i(B)+b_n(F)\cdot b^{(2)}_{i-n}(B).\]
    \item\label{item. l2 betti 3} If $B$ is a closed aspherical manifold of odd dimension and satisfies the Singer Conjecture, then
    \[b^{(2)}_i(E)=0.\]
    
    \item\label{item. l2 betti 4} If $B$ is a closed aspherical manifold with $\dim B=2n$ that satisfies the Singer Conjecture, then
    \[b^{(2)}_i(E)=b_{i-n}(F) \cdot b^{(2)}_n(B).\]
    If in addition $B$ is a closed negatively curved Riemannian manifold, then for every $i$ such that $b_i(F)>0$, we have
    \[b^{(2)}_{n+i}(E)>0.\]
    In particular, if $F$ is a closed manifold, then
    \[b^{(2)}_{n+\dim F}(E)>0.\]
\end{enumerate}
\end{theorem}

We prove  this theorem  by  giving an affirmative answer to the following Question \ref{q. luck} due to L\"{u}ck for locally indicable groups.

\begin{question}[L\"{u}ck]\label{q. luck}
Let $G$ be a group, $\sigma\colon G\rightarrow \GL_n(\C)$ a complex representation of $G$, and $C_{\ast}$ a chain complex of $\Z G$-modules. Is it true that
\[b^{(2),\sigma}_i(C_{\ast})=n\cdot b^{(2)}_i(C_{\ast})\]
for all $i$?
\end{question}

To answer Question \ref{q. luck}, we interpret (twisted) $\ell^2$-Betti numbers as special cases of agrarian Betti numbers and then we extend every complex representation of a locally indicable group $G$ to a representation of $\D_G$. For the rest of this section, let $G$ be a locally indicable group, $\D_G$ the Linnell skew field of $G$, $\tau\colon \C G\hookrightarrow \D_G$ the natural inclusion, and $\sigma\colon G \rightarrow \mathrm{GL}_n(\C)$ a complex representation. By Example \ref{eg. l2 betti}, the (twisted) $\ell^2$-Betti numbers are special cases of agrarian Betti numbers:
\begin{align*}
b^{(2)}_i(C_{\ast}) &=\dim_{\D_G} H_i(C_{\ast}\otimes_{\Z G} \D_G)=b^{\tau}_i(C_{\ast}),\\
b^{(2),\sigma}_i(C_{\ast})& =\dim_{\D_G} H_i(C_{\ast}\otimes_{\Z G} \D_G^n)=b^{\sigma\otimes_{\C} \tau}_i(C_{\ast}) 
\end{align*}
for all $i$. So L\"{u}ck's question will be answered if we can relate $b^{\sigma\otimes_{\C} \tau}_i(C_{\ast})$ to $b^{\tau}_i(C_{\ast})$.

\begin{theorem}\label{thm. extension}
Let $G$ be a locally indicable group, $\tau\colon G\rightarrow \D_G$ the natural map of $G$ into its Linnell skew field, and $\sigma\colon G\rightarrow \GL_n(\C)$ a finite-dimensional complex representation. Then $\sigma\otimes_{\C} \tau$ extends to a ring homomorphism $\ts\colon \D_G\rightarrow M_n(\D_G)$.
\end{theorem}

\begin{proof}
For all $H\leqslant G$, let $\widetilde \D_H$ be the division closure of $(\sigma\otimes_{\C}\tau)(\C H)$ in $M_n(\D_G)$. Let $\rat(\C^{\times}G)$ be the universal rational $\C^{\times}G$-semiring. Section \ref{sec. complexity} gives us a map
\[\Phi\colon \rat(\C^{\times}G)\cup\{0\}\rightarrow \widetilde\D_G.\]
By Remark \ref{rm. complexity} and Lemma \ref{lem. complexity wd}, $\Phi(\rat(\C^{\times}H)\cup\{0\})=\widetilde\D_H$, where we think of $\rat(\C^{\times}H)\cup\{0\}$ as a subset of $\rat(\C^{\times}G)\cup\{0\}$.

Let $\mathcal{T}$ be the set of finite rooted trees. Section \ref{sec. complexity} gives us the notion of $G$-complexity $\tree_G(x)\in \mathcal{T}$ of $x\in\widetilde\D_G$. We first prove that $\widetilde\D_G$ is a skew field by inducting on the $G$-complexity. Our proof uses the idea of the proof of \cite{jaikin2020strong}*{Theorem 6.1}.

Consider a non-zero element $x\in\widetilde\D_G$. If $\tree_G(x)=1_{\mathcal{T}}$, then $x\in(\sigma\otimes_{\C}\tau)(\C^{\times}G)$ is invertible. Now assume that $\tree_G(x)>1_{\mathcal{T}}$ and that for all $0\neq y\in\widetilde \D_G$ with $\tree_G(y)<\tree_G(x)$, $y$ is invertible in $\widetilde\D_G$. Take $\alpha\in\rat(\C^{\times}G)$ realizing the $G$-complexity of $x$. By Theorem \ref{thm. source} \ref{item. source 1} we may assume that $\alpha$ is primitive because multiplying by an element in $\C^{\times}G$ does not change the complexity nor the conclusion about the invertibility of $x$. Set $H$ to be image of $\mathrm{source}(\alpha)$ under the homomorphism $\C^{\times}G\rightarrow \C^{\times}G/\C^{\times}=G$. Then $\alpha\in \rat(\C^{\times}H)$ and so $x\in \Phi(\rat(\C^{\times}H))=\widetilde\D_H$. If $H=\{1\}$, then $\widetilde\D_H=\C\cdot \mathrm{Id}$ and since $x\neq 0$, it is invertible. If $H\neq\{1\}$, then by Theorem \ref{thm. source} \ref{item. source 1}, $H$ is finitely generated, and thus there exists a normal subgroup $N\lhd H$ and an element $t\in H$ of infinite order such that $H=N\rtimes \langle t \rangle$.

Consider the $H$-complexity $\tree_H$ given by Section \ref{sec. complexity}. Note that for all $0\neq y\in\widetilde \D_H$ with $\tree_H(y)<\tree_H(x)$, we have
\begin{equation}\label{eq. complexity}
    \tree_G(y)\leqslant\tree_H(y)<\tree_H(x)=\tree_G(x).
\end{equation}
Indeed, by definition we have $\tree_G(y)\leqslant \tree_H(y)$. Note that $\tree_G(x)=\tree(\alpha)$, where the latter is computed by thinking of $\alpha$ as an element of $\rat(\C^{\times}G)$. Note also that $\tree_H(x)\leqslant \tree(\alpha)$, where the latter is computed by thinking of $\alpha$ as an element of $\rat(\C^{\times}H)$. As pointed out by Remark \ref{rm. complexity}, the two ways of computing $\tree(\alpha)$ yield the same answer. Thus, we also have $\tree_H(x)\leqslant \tree_G(x)$. The induction hypothesis together with \eqref{eq. complexity} then says that  $y$ is invertible in $\widetilde\D_G$.

For simplicity, denote $\tau(t)$ by $t$ and $\sigma(t)\cdot \tau(t)$ by $s$. Let 
\[\D_N(\!(t)\!), \quad \widetilde\D_N(\!(s)\!), \quad M_n(\D_N)(\!(s)\!)\]
be the twisted Laurent power series rings given by Section \ref{sec. complexity}. By Proposition \ref{prop. laurent and complexity} we have $x\in\widetilde\D_N(\!(s)\!)$. So $x$ can be written as a Laurent power series $x=\sum_i x_is^i$ with $x_i\in \widetilde\D_N$. We claim that there are at least two non-zero summands in $\sum_i x_is^i$. Otherwise, we would have $\alpha\in\rat(\C^{\times}N)t^i$ for some $i$, and so 
$\mathrm{source}(\alpha)\subset \C^{\times}N$, and hence $H\leqslant N$, a contradiction.

Thus, Proposition \ref{prop. laurent and complexity} implies that $\tree_H(x_i)<\tree_H(x)$ for all $i$. Inequality \eqref{eq. complexity} implies that 
\[\tree_G(x_i)<\tree_G(x).\]
Thus, the induction hypothesis implies that if $x_i\neq 0$ then it is invertible in $\widetilde\D_G$, and thus in $M_n(\D_N)$, by \eqref{eq. subring}. So $x$ is invertible in $M_n(\D_N)(\!(s)\!)=M_n(\D_N(\!(t)\!))$. Note that $x$ belongs to the subring $M_n(\D_H)$. So $x$ is invertible in $M_n(\D_H)$, and thus in $M_n(\D_G)$. So $x$ is invertible in $\widetilde\D_G$.
Therefore, $\widetilde \D_G$ is a skew field.

We will now show that $\widetilde \D_G$ is a Hughes-free $\C G$-field. Let $H'\leqslant G$ be any non-trivial finitely generated subgroup and suppose $H'=N'\rtimes \langle t' \rangle$ for some normal subgroup $N'\lhd H'$ and an infinite-order element $t'\in H'$. Since $\D_G$ is a Hughes-free $\C G$-field, the sum 
\[M_n(\D_{N'})+M_n(\D_{N'})\cdot \tau(t')+\cdots+M_n(\D_{N'})\cdot \tau(t')^f\]
is direct for every $f\in\mathbb{N}^+$. The containment \eqref{eq. subring} then implies that the sum 
\[\widetilde \D_{N'}+\widetilde \D_{N'}\cdot \sigma(t')\tau(t')+\cdots+\widetilde \D_{N'}\cdot \sigma(t')^f\tau(t')^f\]
is also direct, and thus $\widetilde \D_G$ is a Hughes-free $\C G$-field. Theorem \ref{thm. hughes} then implies that there exists a ring homomorphism $\ts\colon \D_G\rightarrow M_n(\D_G)$ that extends $\sigma\otimes_{\C}\tau$.
\end{proof}

The map $\ts$ endows $\D_G^n = \oplus_n \D_G$ with a left $\D_G$-module structure by $c\bullet v=\ts(c)\cdot v$. Given any right $\D_G$-module $U$, the action $\bullet$ then induces a tensor product $U\otimes_{\D_G} \D_G^n$, $(u\cdot c)\otimes v=u\otimes(c\bullet v)=u\otimes(\ts(c)\cdot v)$ for all $u\in U,c\in\D_G,v\in \D^n_G$. Viewing $\D_G^n$ as a right $\D_G$-module with the obvious right $\D_G$-action, we see that $U\otimes_{\D_G} \D_G^n$ is then equipped with a structure of a right $\D_G$-module, from which we can calculate the $\D_G$-dimension. Note that
\begin{equation}\label{eq. dim}
    \dim_{\D_G} U\otimes_{\D_G} \D_G^n = n\cdot \dim_{\D_G} U.
\end{equation}
Here we adopt the convention $0\cdot \infty=0$ and $n\cdot \infty=\infty$ for all $n>0$.

\begin{theorem}\label{thm. twisted l2 betti}
Let $G$ be a locally indicable group, let $C_{\ast}$ be a $\Z G$-chain complex, and let $\sigma\colon G\rightarrow \GL_n(\C)$ be a linear representation of $G$. Then for all $i$ we have
\[b^{(2),\sigma}_i(C_{\ast})=n\cdot b^{(2)}_i(C_{\ast}).\]
\end{theorem}

\begin{proof}
Identify
\[C_{\ast}\otimes_{\Z G}\D_G^n\cong C_{\ast}\otimes_{\Z G} \D_G \otimes_{\D_G} \D_G^n.\]
As a left-module over the division ring $\D_G$ (with the action given by $\bullet$), $\D_G^n$ is free and thus for all $i$
\[H_i(C_{\ast}\otimes_{\Z G} \D_G \otimes_{\D_G} \D_G^n)\cong H_i(C_{\ast}\otimes_{\Z G} \D_G)\otimes_{\D_G} \D_G^n.\]
The desired result then follows from \eqref{eq. dim}.
\end{proof}

Question \ref{q. luck} arises naturally in the process of computing $\ell^2$-Betti numbers of fibrations. The following argument can be easily extracted from the proof of \cite{lueck2018twisting}*{Lemma 5.4}. We reproduce it here for the convenience of the reader. Let $F\rightarrow E\rightarrow B$ be a fibration of connected finite CW-complexes, or more generally, topological spaces that are homotopy equivalent to connected finite CW-complexes, such that $\pi_1(B)$ is locally indicable and the induced homomorphism $\pi_1(E)\rightarrow \pi_1(B)$ is bijective (e.g., when $F$ is simply connected). The Leray--Serre spectral sequence then yields
\begin{equation}\label{eq. spectral sequence}
    \small E^2_{p,q}=H_p(C_{\ast}(\widetilde{B})\otimes_{\Z [{\pi_1(B)}]}(H_q(F,\C)\otimes_{\C} \D_{\pi_1(B)}))\Rightarrow H_{p+q}(C_{\ast}(\widetilde{E})\otimes_{\Z [{\pi_1(E)}]}\D_{\pi_1(B)}),
\end{equation}
where $\Z\pi_1(B)$ acts on $H_q(F,\C)\otimes_{\C} \D_{\pi_1(B)}$ by the diagonal action and $\Z\pi_1(E)$ acts on $\D_{\pi_1(B)}$ via the induced isomorphism $\pi_1(E)\cong \pi_1(B)$. Thus, $E^2_{p,q}$ is the $\ell^2$-homology of $B$ twisted by the representation $\eta\colon B\rightarrow \mathrm{GL}(H_q(F,\C))$. Theorem \ref{thm. twisted l2 betti} then implies
\begin{equation}\label{eq. E2 of spectral}
    \dim_{\D_{\pi_1(B)}}E^2_{p,q}=b^{(2),\eta}_p(B)=b_q(F)\cdot b^{(2)}_p(B).
\end{equation}

Below, we prove Theorem \ref{thm. l2 betti fibration}, and Corollaries \ref{cor. surface base} and \ref{cor. 3-mnfl base}.

\begin{proof}[Proof of Theorem \ref{thm. l2 betti fibration}]
Suppose first that $\pi_1(B)$ is locally indicable. Since $\D_{\pi_1(B)}$ is a skew field, the spectral sequence \eqref{eq. spectral sequence} implies that $H_n(C_{\ast}(\widetilde{E})\otimes_{\Z [{\pi_1(E)}]}\D_{\pi_1(B)})$ is a a direct sum of subquotients of $E^2_{i,n-i}$ for $i=0,1,\cdots,n$. Together with \eqref{eq. E2 of spectral}, this implies \ref{item. l2 betti 1}.

Suppose that the assumption of \ref{item. l2 betti 2} holds. Then Theorem \ref{thm. twisted l2 betti} implies that the spectral sequence \eqref{eq. spectral sequence} stabilizes at the $E^2$-page with
\[
\dim_{\D_{\pi_1(B)}} E^2_{p,q}=
\begin{cases}
b^{(2)}_p(B),&\text{if }q=0\\
b_n(F)\cdot b^{(2)}_p(B),&\text{if }q=n\\
0,&\text{otherwise.}
\end{cases}
\]
Item \ref{item. l2 betti 2} follows from computation using the spectral sequence \eqref{eq. spectral sequence}.

Suppose that the assumption of \ref{item. l2 betti 3} holds. Then the Singer Conjecture implies that $b^{(2)}_{\ast}(B)=0$. Theorem \ref{thm. twisted l2 betti} implies that the $E^2$-page of the spectral sequence \eqref{eq. spectral sequence} is $0$, from which \ref{item. l2 betti 3} follows.
 
Suppose that the assumption of \ref{item. l2 betti 4} holds. Then Theorem \ref{thm. twisted l2 betti} implies that the spectral sequence \eqref{eq. spectral sequence} stabilizes at the $E^2$-page with
\[
\dim_{\D_{\pi_1(B)}} E^2_{p,q}=
\begin{cases}
b_q(F)\cdot b^{(2)}_n(B), &\text{if } p=n\\
0, &\text{otherwise.}
\end{cases}
\]
Item \ref{item. l2 betti 4} follows from computation using the spectral sequence \eqref{eq. spectral sequence}. This finishes the proof for the special case where $\pi_1(B)$ is locally indicable.

\smallskip

Let us consider the general case where $\pi_1(B)\cong \pi_1(E)$ are virtually locally indicable. Let $\widehat B$ be a $d$-sheeted cover of $B$ for some $d$ such that $\pi_1(\widehat B)$ is locally indicable, and let $\widehat E$ be the pullback of $E\rightarrow B$ along $\widehat B\rightarrow B$. Then we have a fibration
\[F\rightarrow \widehat E \rightarrow \widehat B\]
with $\pi_1(\widehat E)=\pi_1(\widehat B)$ locally indicable. 

Note that $\dim B\leqslant n$ if and only if $\dim \widehat B\leqslant n$, $B$ is a closed aspherical manifold if and only if so is $\widehat B$, and $B$ is a Riemannian manifold with negative sectional curvature if and only if so is $\widehat B$. By the above, items \ref{item. l2 betti 1}, \ref{item. l2 betti 2}, \ref{item. l2 betti 3}, \ref{item. l2 betti 4} hold with $\widehat E$ in place of $E$ and $\widehat B$ in place of $B$. By \cite{luck2002l2}*{Theorem 1.35 (9)}, we have for all $i$
\[b^{(2)}_i(\widehat E)=d\cdot b^{(2)}_i(E),~~b^{(2)}_i(\widehat B)=d\cdot b^{(2)}_i(B),\]
which finishes the proof.
\end{proof}

\begin{proof}[Proof of Corollary \ref{cor. surface base}]
In this case, the spectral sequence \eqref{eq. spectral sequence} has only one non-zero column, and thus stabilizes. By Example \ref{eg. locally indicable}, $\pi_1(B)$ is locally indicable. Thus, the desired result follows from Theorem \ref{thm. twisted l2 betti} and 
\[
b^{(2)}_i(B)=
\begin{cases}
-\chi(B),&\text{if }i=1\\
0,&\text{otherwise.}
\end{cases}\qedhere
\]
\end{proof}

\begin{proof}[Proof of Corollary \ref{cor. 3-mnfl base}]
By Lemma \ref{lem. eg of locally indicable} $\pi_1(B)$, is virtually locally indicable. Thus, the corollary follows from Theorem \ref{thm. l2 betti fibration} \ref{item. l2 betti 1} and the computation of the $\ell^2$-Betti number of $3$-manifolds \cite{lott1995l2}*{Theorem 0.1}.
\end{proof}

\section{Agrarian norm and Euler characteristic}\label{sec. euler characteristic}
If $M$ is a closed connected orientable irreducible $3$-manifold and $\phi\in H^1(M,\Z)$ is a character induced by a fibration $F\rightarrow M\rightarrow S^1$ of $M$ over the circle $S^1$, then $\|\phi\|_T=-\chi(F)$, where $\|\phi\|_T$ is the Thurston norm of $\phi$ \cite{thurston1986norm}. The goal of the current section is \cref{thm. euler characteristic} below, which generalizes this result of \cite{thurston1986norm}. In \cref{sec. aspherical group,sec. application}, we will apply \cref{thm. euler characteristic} to deduce the equality between the twisted Alexander and Thurston norms for fibered characters.

\begin{theorem}\label{thm. euler characteristic}
Let $G$ be a (type $F$)-by-(infinite cyclic) group and let $\phi\in H^1(G,\Z)$ be a primitive character. Then for every linear representation $\sigma\colon G\rightarrow \GL_n(D)$ over a skew field $D$, $\|\phi\|_{\sigma}$ and $\chi^{\sigma\otimes_{\Z} q}(\ker\phi)$ are well-defined and
\[\|\phi\|_{\sigma}=-\chi^{\sigma\otimes_{\Z} q}(\ker\phi),\]
where $q\colon G\rightarrow G_{\mathrm{fab}}$ is the natural quotient map from $G$ onto its maximal free abelian quotient $G_{\mathrm{fab}}$.
\end{theorem}

Below, we use the notation of the above theorem. Let $H=\ker \phi$, let $t\in G$ such that $\phi(t)=1$, let $\bar t=q(t)$, let $X$ be a basis of $G_{\mathrm{fab}}$ such that $\bar t\in X$ and $\phi(x)=0$ for all $x\in X\smallsetminus\{\bar t\}$, let $BG$ be a finite $K(G,1)$ CW-complex, let $EG$ be the universal cover of $BG$, let $\bar q\colon G\twoheadrightarrow G_{\mathrm{fab}}/\langle \bar t\rangle$ the natural quotient map, and let $Y=X\smallsetminus\{\bar t\}$. To emphasize the role played by $X$ and $Y$, we denote $\ore(DG_{\mathrm{fab}})$ by $D(X)$ and $\ore(DL)$ by $D(Y)$, where $L$ is the subgroup of $G_{\mathrm{fab}}$ generated by $Y$. 

For all $k$,
\[\dim_{D(X)}H_k(C_{\ast}(EG)\otimes_{\Z H} (D(X))^n)=\dim_{D(Y)}H_k(C_{\ast}(EG)\otimes_{\Z H} (D(Y))^n)\]
as $D(X)$ is flat over $D(Y)$, where $(C_{\ast}(EG),\partial^{EG}_{\ast})$ is the cellular chain complex of $EG$ and the right-hand side tensor product is taken with respect to the representation $\sigma\otimes_{\Z} \bar q$. Therefore,
\begin{equation}\label{eq. bar q}
    \chi^{\sigma\otimes_{\Z} q}(H)=\chi^{\sigma\otimes_{\Z} \bar q}(H).
\end{equation}

Note that there is an isomorphism of chain complexes of $D(Y)$-modules
\begin{equation}\label{eq. [t]}
    C_{\ast}(EG)\otimes_{\Z H} (D(Y))^n \xrightarrow[]{\cong} C_{\ast}(EG)\otimes_{\Z G} (D(Y)[\bar t^{\pm}])^n,
\end{equation}
which maps $e\otimes d$ to $e\otimes d$ for all $e\in C_{\ast}(EG)$ and $d\in (D(Y))^n$. The inverse of this map sends $e\otimes d\bar t^k$ to $(et^k)\otimes (\sigma(t))^{-k}(d)$. Note also the following isomorphism of chain complexes of $D(X)$-modules
\begin{equation}\label{eq. (t)}
    C_{\ast}(EG)\otimes_{\Z G} (D(Y)[\bar t^{\pm}])^n\otimes_{D(Y)[\bar t^{\pm}]}D(X)\cong C_{\ast}(EG)\otimes_{\Z G} (D(X))^n.
\end{equation}

For simplicity, let 
\[(C_{\ast},\partial_{\ast})=(C_{\ast}(EG)\otimes_{\Z G} (D(Y)[\bar t^{\pm}])^n,\partial^{EG}_{\ast}\otimes_{\Z G} \mathrm{id}_{(D(Y)[\bar t^{\pm}])^n}).\]
For $k\in\mathbb{N}$, let $\B^{EG}_k$ be a $\Z G$-basis of $C_k(EG)$ consisting of cells of dimension $k$, and let $V$ be the standard $D(Y)[\bar t^\pm]$-basis of $(D(Y)[\bar t^{\pm}])^n$. Then 
\[\B_k=\{\Delta\otimes v\mid \Delta\in \B^{EG}_k,v\in V\}\subset C_{\ast}\]
is a $D(Y)[\bar t^{\pm}]$-basis for $C_{\ast}$. By definition, $\|\phi\|_{\sigma}$ is computed using the basis $\B_{\ast}$. But in order to prove the theorem we will use another basis that is equivalent to $\B_{\ast}$.

\begin{lemma}\label{lem. change of basis}
There are two families of subsets of $C_{\ast}$, $\{\B'_k\}^{\infty}_{k=0}$ and $\{\B''_k\}^{\infty}_{k=0}$, such that the following hold for every $k$.
\begin{enumerate}[label=(\roman*)]
    \item\label{item. change basis 4} $\B'_k\cup\B''_k$ is a basis of $C_k$ and the change of basis matrix $M$ from $\B_k$ to $\B'_k\cup\B''_k$ satisfies $\Det_{D(X)}(M)=\pm 1$, where we think of $M$ as a matrix over $D(X)$ to take the determinant.
    \item\label{item. change basis 2} Denote by $\spanY$ the linear span over $D(Y)[\bar t^{\pm}]$. Then
    \begin{align*}
        \partial_k(\spanY(\B'_k))&\subset \spanY (\B''_{k-1}),\\
        \spanY(\B''_k)&=\ker \partial_k.
    \end{align*}
    \item\label{item. change basis 3} Let $\overline \partial_k\colon \spanY(\B'_k)\rightarrow \spanY (\B''_{k-1})$ be the restriction of $\partial_k$. Then the matrix representative of $\overline \partial_k$ under the bases $\B'_k$ and $\B''_{k-1}$, denoted $[\overline \partial_k]$, is a full-rank diagonal matrix over $D(Y)[\bar t^{\pm}]$.
\end{enumerate}
\end{lemma}

\begin{proof}
We prove the lemma by an induction on $k$. First note that \ref{item. change basis 4} through \ref{item. change basis 3} hold for $k=0$ with $\B'_0=\emptyset,\B''_0=\B_0$. Now suppose that we have found $\{\B'_k\}^K_{k=0}$ and $\{\B''_k\}^K_{k=0}$ that satisfy \ref{item. change basis 4} through \ref{item. change basis 3} for $k\leqslant K$. Let $M_{K+1}$ be the matrix representative of $\partial_{k+1}$ under the bases $\B_{K+1}$ and $\B'_K\cup\B''_K$.

By multiplying $M_{K+1}$ on the left and right by elementary matrices over $D(Y)[\bar t^{\pm}]$ whose diagonal entries are $\pm 1$, we can turn $M_{K+1}$ into a diagonal matrix $N_{K+1}$ over $D(Y)[\bar t^{\pm}]$. The left (resp. right) multiplication of elementary matrices corresponds to the change of the basis $\B'_K\cup \B''_K$ (resp. $\B_{K+1}$). Since non-zero elements of $\spanY(\B'_K)$ have non-zero boundaries, we have $\partial_{K+1}(C_{K+1})\subset \spanY(\B''_K)$. Therefore, we may assume that the change of basis process leaves $\B'_K$ invariant and turns $\B''_K$ into another basis of $\spanY(\B''_K)$. Since $\partial_K(\spanY(\B''_K))=0$, such a change of basis process will not change $[\overline\partial_K]$. Thus, by modifying $\B''_K$ if necessary, we may assume that the change of basis process leaves both $\B'_K$ and $\B''_K$ invariant.

Let $\B'_{K+1}$ (resp. $\B''_{K+1}$) be the part of the new basis of $C_{K+1}$ corresponding to the non-zero (resp. zero) diagonal entries of $N_{K+1}$. Then \ref{item. change basis 4} and \ref{item. change basis 2} follow immediately. Item \ref{item. change basis 3} is equivalent to $|\B'_{K+1}|=|\B''_K|$. Consider
\[C_{K+1}\otimes_{D(Y)[\bar t^{\pm}]}D(X)\xrightarrow[]{\partial_{K+1}\otimes_{D(Y)[\bar t^{\pm}]} \mathrm{id}_{D(X)}}C_K\otimes_{D(Y)[\bar t^{\pm}]}D(X).\]
And consider the subsets
\begin{align*}
\B'_{K+1}\otimes\{1\}, \B''_{K+1}\otimes\{1\} &\subset C_{K+1}\otimes_{D(Y)[\bar t^{\pm}]}D(X),\\
\B_K\otimes\{1\}, \B''_K\otimes\{1\} &\subset C_K\otimes_{D(Y)[\bar t^{\pm}]}D(X).
\end{align*}
Since $D(X)$ is flat over $D(Y)[\bar t^{\pm}]$, $(\B'_{K+1}\cup\B''_{K+1})\otimes\{1\}$ (resp. $(\B'_K\cup\B''_K)\otimes\{1\}$) is a $D(X)$-basis of $C_{K+1}\otimes_{D(Y)[\bar t^{\pm}]}D(X)$ (resp. $C_K\otimes_{D(Y)[\bar t^{\pm}]}D(X)$). 

Let $N$ be the matrix representative of $\partial_{K+1}\otimes_{D(Y)[\bar t^{\pm}]} \mathrm{id}_{D(X)}$ under the bases $(\B'_{K+1}\cup\B''_{K+1})\otimes\{1\}$ and $(\B'_K\cup\B''_K)\otimes\{1\}$. Then $N$ is the $(|\B'_{K+1}|+|\B''_{K+1}|)\times(|\B'_K|+|\B''_K|)$-matrix with $[\overline\partial_{K+1}]$ at the top left corner and $0$ elsewhere. In particular,
\[|\B'_{K+1}|=\rank_{D(X)} (\partial_{K+1}\otimes_{D(Y)[\bar t^{\pm}]} \mathrm{id}_{D(X)}).\]
Similarly,
\[|\B'_K|=\rank_{D(X)} (\partial_K\otimes_{D(Y)[\bar t^{\pm}]} \mathrm{id}_{D(X)}),\]
and thus
\[|\B''_K|=\dim_{D(X)}\ker (\partial_K\otimes_{D(Y)[\bar t^{\pm}]} \mathrm{id}_{D(X)}).\] 
Since $G$ is (type $F$)-by-(infinite cyclic), $C_{\ast}(EG)\otimes_{D(Y)[\bar t^{\pm}]} (D(X))^n$ is acyclic by \cref{cor. fxz}. Combining with equation \eqref{eq. (t)} this yields
\[\rank_{D(X)}(\partial_{K+1}\otimes_{D(Y)[\bar t^{\pm}]} \mathrm{id}_{D(X)})=\dim_{D(X)}\ker (\partial_K\otimes_{D(Y)[\bar t^{\pm}]} \mathrm{id}_{D(X)}),\]
from which \ref{item. change basis 3} follows.
\end{proof}

\begin{proof}[Proof of Theorem \ref{thm. euler characteristic}]
Fix $k\in\mathbb{N}$. Let $e_1,\cdots, e_{\ell_k}$ be the elements of $\B''_k$ and let $f_1,\cdots, f_{\ell_k}$ be the diagonal entries of $[\overline\partial_{k+1}]$. For $i=1,2,\cdots, \ell_k$, let
\[S_{k,i}=\{e_i\cdot \bar t^j \mid j=0,1,\cdots, (\deg_{\bar t}f_i)-1\}.\]
and let
\[S_k=\bigcup^{\ell_k}_{i=1}S_{k,i}\subset C_k.\]
Then there is a $D(Y)$-module isomorphism,
\[\spanY(\B''_k)/\im \overline\partial_{k+1}\cong \mathrm{span}_{D(Y)}(S_k).\]
By \eqref{eq. [t]},
\[b^{\sigma\otimes_{\Z}\bar q}_k(H)=\dim_{D(Y)}(\spanY(\B''_k)/\im \overline\partial_{k+1})=\deg_{\bar t}\Det_{D(X)} [\overline\partial_{k+1}],\]
where we think of $[\partial_{k+1}]$ as a matrix over $D(X)$ in order to take the determinant. It then follows from \eqref{eq. bar q} that
\begin{equation}\label{eq. euler characteristic 1}
    \chi^{\sigma\otimes_{\Z} q}(H)=\sum^{\infty}_{k=0}(-1)^k \deg_{\bar t}\Det_{D(X)} [\partial_{k+1}].
\end{equation}

 On the other hand, by Lemma \ref{lem. change of basis} \ref{item. change basis 4} and \ref{item. change basis 3}, $C_{\ast}$ decomposes as a direct sum of chain complexes of the form
\[0\rightarrow \spanY(\B'_k)\xrightarrow[]{[\overline\partial_k]}\spanY(\B''_{k-1})\rightarrow 0.\]
By tensoring with $D(X)$ we see that $C_{\ast}(EG)\otimes_{\Z G} (D(X))^n$ decomposes as the direct sum of chain complexes of the form
\[0\rightarrow \spanX(\B'_k\otimes\{1\})\xrightarrow[]{[\overline\partial_k]}\spanX(\B''_{k-1}\otimes\{1\})\rightarrow 0,\]
where we think of $[\partial_k]$ as a matrix over $D(X)$.

By Lemma \ref{lem. change of basis} \ref{item. change basis 4} and Remark \ref{rm. change of basis},
\begin{equation}\label{eq. euler characteristic 2}
    \|\phi\|_{\sigma}=\sum^{\infty}_{k=0}(-1)^k\deg_{\bar t}\Det_{D(X)} [\partial_k].
\end{equation}

The desired result follows from equations \eqref{eq. euler characteristic 1} and \eqref{eq. euler characteristic 2}.
\end{proof}

\section{Aspherical groups}\label{sec. aspherical group}
In this section, we prove the semi-norm property of the agrarian norm and the inequality between the twisted Alexander and Thurston norms for certain aspherical groups. Let $G$ be a finitely presentated group and let $q\colon G\rightarrow G_{\mathrm{fab}}$ be the natural homomorphism of $G$ onto its maximal free abelian quotient $G_{\mathrm{fab}}$. We start with a method to modify a given finite group presentation.

\begin{lemma}\label{lem. good presenation}
Let
$G$ be a group given by a finite presentation
\begin{equation}\label{eq. presentation 1}
    G=\langle X\mid \mathcal{R}\rangle.
\end{equation}
Then there exists a finite presentation
\begin{equation}\label{eq. presentation 2}
    G=\langle Y\mid \mathcal{S}\rangle.
\end{equation}
such that $q(Y)\smallsetminus\{0\}$ is a basis for $G_{\mathrm{fab}}$ and the presentation complexes of \eqref{eq. presentation 1}, \eqref{eq. presentation 2} are homotopy equivalent.

Moreover, if $\phi\in H^1(G,\Z)$ is a primitive integral character, then we can further guarantee that there exists $y\in Y$ such that $\phi(y)=1$ and $\phi(y')=0$ for all $y'\in Y\smallsetminus\{y\}$.
\end{lemma}

\begin{proof}
For $x_i,x_j\in X$, by replacing $x_i$ with $x_ix_j$ or $x_ix^{-1}_j$ and doing the corresponding replacement among the relations of $\mathcal{R}$ that contain $x_i$, we obtain a presentation $G=\langle X'\mid \mathcal{R}'\rangle$. We call the passage from $\langle X\mid \mathcal{R}\rangle$ to $\langle X'\mid \mathcal{R}'\rangle$ a \textit{Nielsen transformation}. Let $K$ (resp. $K'$) be the presentation complex of $\langle X\mid \mathcal{R}\rangle$ (resp. $\langle X'\mid \mathcal{R}'\rangle$). Then to pass from $K$ to $K'$, one can subdivide the edge labeled by $x_i$ into two edges and identify one of these new edges with $x_j$ or $x^{-1}_j$, from which it is not hard to see that $K$ and $K'$ are homotopy equivalent, and thus $K'$ also has contractible universal cover. Now, starting from $\langle X\mid \mathcal{R}\rangle$ and inductively performing Nielsen transformations, we can obtain the desired presentation.
\end{proof}

For the rest of this section, suppose $G$ is a semi-direct product $H\rtimes \Z$ with $H$ a type $F$ subgroup. Let $\sigma\colon G\rightarrow \GL_n(D)$ be a representation over a skew field $D$. Then $G$ is $(\sigma\otimes_{\Z} q)$-acyclic by Proposition \ref{prop. mapping torus and vanishing betti number}, and thus the agrarian norm $\|\cdot\|_{\sigma}$ is well-defined. If $\rank (G_{\mathrm{fab}})=0$ then all agrarian norms are trivial. So below we assume that $\rank(G_{\mathrm{fab}})\geqslant 1$. 

\subsection{Semi-norm property}
\begin{lemma}\label{lem. rkGfb=1 case}
If $G_{\mathrm{fab}}\cong \Z$ then $\|\cdot\|_{\sigma}$ is a semi-norm if and only if $\chi(H)\leqslant 0$.
\end{lemma}

\begin{proof}
The desired conclusion follows from Theorem \ref{thm. euler characteristic} and Proposition \ref{prop. euler characteristic}.
\end{proof}

Next, let us further suppose that $G$ is \textit{aspherical}, i.e., $G$ has a finite presentation 
\begin{equation}\label{eq. good presentation 3}
    G=\langle X\mid \mathcal{R}\rangle
\end{equation}
such that the corresponding presentation complex has contractible universal cover. By Lemma \ref{lem. good presenation} we may assume that $X=\{x_i\}^{k+m}_{i=1}$ with $\{q(x_i)\}^k_{i=1}$ being a basis for $G_{\mathrm{fab}}$ and $q(x_{k+1})=q(x_{k+2})=\cdots=q(x_{k+m})=0$.

\begin{proposition}\label{prop. semi-norm}
Suppose that $G$ is a group that is (type $F$)-by-(infinite cyclic) and aspherical, and satisfies $\rank G_{\mathrm{fab}}\geqslant 2$. Then for every linear representation $\sigma\colon G\rightarrow \GL_n(D)$ of $G$ over a skew field $D$, the function $\|\cdot\|_{\sigma}$ is a semi-norm.
\end{proposition}

\begin{remark}
The above proposition is no longer true if $\rank G_{\mathrm{fab}}\geqslant 2$ is dropped. An easy example is given by $G=\Z$.
\end{remark}

\begin{proof}
Let $K$ be the presentation complex of \eqref{eq. good presentation 3}. Consider the cellular chain complex of $\widetilde{K}$, the universal cover of $K$.

\begin{tikzcd}
C_{\ast}\colon & 0 \arrow[r] & \Z G^{k+m-1} \arrow[r,"M_2"] & \Z G^{k+m} \arrow[r,"M_1"] & \Z G \arrow[r] & 0,
\end{tikzcd}

\noindent
where $M_2$ is a $(k+m)\times(k+m-1)$ matrix over $\Z G$ and $M_1=(1-x_1,\cdots,1-x_{k+m})$. Note that we must have $\Z G^{k+m-1}$ in dimension $2$ as $\chi(G)=0$ (see, e.g., Proposition \ref{prop. euler characteristic} and Corollary \ref{cor. fxz}). The corresponding maps are given by the left multiplication by $M_2$ and $M_1$.

Let

\begin{tikzcd}
A_{\ast}\colon  & 0 \arrow[r] & \Z G^{k+m-1} \arrow[r,"M'_2"] & \Z G^{k+m-1} \arrow[r] & 0 \arrow[r] & 0,\\
B_{\ast}\colon  & 0 \arrow[r] & 0 \arrow[r] & \Z G \arrow[r,"1-x_1"] & \Z G \arrow[r] & 0,
\end{tikzcd}

\noindent
where $M'_2$ is the matrix obtained from $M_2$ by deleting the first row. We then have a short exact sequence of $\Z G$-chain complexes
\begin{equation}\label{eq. exact for aspherical}
   0\rightarrow A_{\ast}\rightarrow C_{\ast} \rightarrow B_{\ast} \rightarrow 0.
\end{equation}

Let $Y=\{q(x_i)\}^k_{i=1}$. We identify $\ore(DG_{\mathrm{fab}})$ with $D(Y)$ to emphasize the role played by $Y$. By tensoring \eqref{eq. exact for aspherical} with $(D(Y))^n$ using the representation $\sigma\otimes_{\Z} q$, we obtain
\begin{equation}\label{eq. exact for ashperical 1}
    0\rightarrow A_{\ast}\otimes_{\Z G} (D(Y))^n\rightarrow C_{\ast}\otimes_{\Z G} (D(Y))^n \rightarrow B_{\ast}\otimes_{\Z G} (D(Y))^n \rightarrow 0.
\end{equation}

Consider the unique non-zero differential of $B_{\ast}\otimes_{\Z G} (D(Y))^n$, which is given by the left multiplication by the matrix
\[(\sigma\otimes_{\Z} q)(1-x_1)=\mathrm{Id}+\sigma(-x_1)q(x_1).\]
Every entry of $\sigma(-x_1)q(x_1)$ has $q(x_1)$-degree at least $1$. So Lemma \ref{lem. matrix reduction} implies that $(\sigma\otimes_{\Z} q)(1-x_1)$ is invertible, and thus the chain complex $B_{\ast}\otimes_{\Z G} (D(Y))^n$ is exact. This, together with the exactness of $C_{\ast}\otimes_{\Z G}(D(Y))^n$ (by Corollary \ref{cor. fxz}), yields that $A_{\ast}\otimes_{\Z G} (D(Y))^n$ is also exact and \eqref{eq. exact for ashperical 1} is a short exact sequence of chain complexes.

Let $\rho_A$ (resp. $\rho_B,\rho_C$) be the Riedemeister torsion of $A_{\ast}\otimes_{\Z G} (D(Y))^n$ (resp. $B_{\ast}\otimes_{\Z G} (D(Y))^n,C_{\ast}\otimes_{\Z G} (D(Y))^n$). Then (see, e.g.,  \cite{cohen1973course}*{(17.2)})
\begin{equation}\label{eq. structure of torsion}
    \rho_C=\rho_A\cdot\rho_B.
\end{equation}

Let $P\colon D(Y)\rightarrow \mathcal{P}(G_{\mathrm{fab}})$ be the polytope homomorphism. Equation \eqref{eq. structure of torsion} implies
\begin{equation}\label{eq. structure of polytope}
    P(\rho_C)=P(\Det_{D(Y)} (\sigma\otimes_{\Z} q)(M'_2))-P(\Det_{D(Y)}(\sigma\otimes_{\Z} q)(1-x_1)).
\end{equation}

Since $M'_2$ is a matrix over $\Z G$, Theorem \ref{thm. single polytope} implies that 
\[P(\Det_{D(Y)} (\sigma\otimes_{\Z} q)(M'_2))\in P(D[Y^{\pm}])\]
is a single polytope, and thus \eqref{eq. structure of polytope} implies
\begin{equation}\label{eq. pre single polytope for aspherical}
    P(\rho_C)\in P(D[Y^{\pm}])-P(D[x^{\pm}_1]).
\end{equation}

By the assumption $\rank G_{\mathrm{fab}}\geqslant 2$ we also have $q(x_2)\in Y$. The above argument with $x_2$ in place of $x_1$ yields
\[P(\rho_C)\in P(D[Y^{\pm}])-P(D[x^{\pm}_2]),\]
which together with \eqref{eq. pre single polytope for aspherical} implies that $P(\rho_C)$ is a single polytope, which in turn yields the desired result.
\end{proof}

\subsection{Inequality between the Alexander and Thurston norms}
The inequality between the Alexander and Thurston norms for $3$-manifolds was discovered by McMullen \cite{mcmullen2002alexander}, whose result was then generalized by Friedl--Kim \cite{friedl2008twisted} and Funke and the first author \cite{funke2018alexander}. In the current and subsequent sections, we recover the result of \cite{friedl2008twisted} and generalize the result of \cite{funke2018alexander}.

\begin{theorem}\label{thm. inequality for apspherical group}
Suppose that $G$ is an aspherical (Lewin type $F$)-by-(infinite cyclic) group. Let $\|\cdot\|_T$ be the Thurston norm of $G$. Then for all finite dimensional complex representations $\sigma\colon G\rightarrow \GL_n(\C)$ and all $\phi\in H^1(G,\Z)$, one has
\[
    \|\phi\|_{\sigma}\leqslant n\cdot \|\phi\|_T.
\]
Moreover, equality holds if $\ker \phi$ is of type $F$.
\end{theorem}

\begin{proof}
Without loss of generality, we may assume that $\phi$ is a primitive integral character. By Lemma \ref{lem. good presenation}, there is a finite presentation
\[G=\langle x_1,\cdots,x_{k+m}\mid \mathcal{R}\rangle\]
such that 
\begin{enumerate}
    \item[(i)] $\{q(x_i)\}^k_{i=1}$ is a basis for $G_{\mathrm{fab}}$;
    \item[(ii)] $q(x_{k+1})=q(x_{k+2})=\cdots=q(x_{k+m})=0$;
    \item[(iii)] $\phi(x_1)=1$ and $\phi(x_2)=\phi(x_3)=\cdots=\phi(x_k)=0$.
\end{enumerate}

Let $\D_G$ be the Linnell skew field of $G$ and let $\tau\colon \C G\rightarrow \D_G$ be the natural embedding. We identify $\ore(\C G_{\mathrm{fab}})$ (resp. $\ore(\D_G G_{\mathrm{fab}})$) with $\C(Y)$ (resp. $\D_G(Y)$) to emphasize the role played by $Y$. For simplicity, we also denote $q(x_1)$ by $s$. Equation \eqref{eq. structure of polytope} implies
\begin{equation}\label{eq. 11}
    \|\phi\|_{\sigma}=\deg_s\Det_{\C(Y)} (\sigma\otimes_{\Z} q)(M'_2)-\deg_s\Det_{\C(Y)} (\sigma\otimes_{\Z} q)(1-x_1).
\end{equation}

Similarly,
\begin{equation}\label{eq. 12}
\|\phi\|_T =\deg_s\Det_{\D_G(Y)} (\tau\otimes_{\Z} q)(M'_2)-\deg_s\Det_{\D_G(Y)} (\tau\otimes_{\Z} q)(1-x_1),
\end{equation}

Since $G$ is Lewin \cite{jaikin2020universality}*{Theorem 3.7 (3)}, \cref{lem. key lem} implies
\begin{equation}\label{eq. 796}
    \deg_s\Det_{\C(Y)} (\sigma\otimes_{\Z} q)(M'_2)
    \leqslant \deg_s\Det_{\D_G(Y)} (\sigma\otimes_{\C}\tau\otimes_{\Z} q)(M'_2).
\end{equation}

Consider the representation $\sigma\otimes_{\C}\tau\colon G\rightarrow \mathrm{GL}_n(\D_G)$. Theorem \ref{thm. extension} extends $\sigma\otimes_{\C}\tau$ to a ring homomorphism $\ts\colon \D_G\rightarrow M_n(\D_G)$. By repeatedly using Corollary \ref{cor. ring homo extension}, we further extends $\ts$ to ring homomorphism $\ts\colon \D_G(Y)\rightarrow M_n(\D_G(Y))$ such that $\ts(y)=\mathrm{Id}\cdot y$ for all $y\in Y$. We have
\[\ts\big((\tau\otimes_{\Z}q)(M'_2)\big)=(\sigma\otimes_{\C}\tau\otimes_{\Z} q)(M'_2).\]

Lemma \ref{lem. extension and degree} thus implies
\begin{equation}\label{eq. 810}
    \deg_s\Det_{\D_G(Y)} (\sigma\otimes_{\C}\tau\otimes_{\Z} q)(M'_2)=n\cdot \deg_s\Det_{\D_G(Y)} (\tau\otimes_{\Z} q)(M'_2).
\end{equation}

Think of $\Det_{\C(Y)}(\sigma\otimes_{\Z} q)(1-x_1)$ as a polynomial in $s$ with coefficient in ${\C(Y\smallsetminus\{s\})}$. Then the highest power of $s$ in $\Det_{\C(Y)}(\sigma\otimes_{\Z} q)(1-x_1)$ is $s^n$ with coefficient $\Det_{\C(Y)}\sigma(-x_1)$. The lowest power of $s$ in $\Det_{\C(Y)}(\sigma\otimes_{\Z} q)(1-x_1)$ is $s^0=1$ with coefficient $1$. Thus,
\begin{equation}\label{eq. 92146}
    \deg_s\Det_{\C(Y)} (\sigma\otimes_{\Z} q)(1-x_1)=n=n\cdot \deg_s\Det_{\D_G(Y)} (\tau\otimes_{\Z} q)(1-x_1).
\end{equation}

We conclude from \eqref{eq. 11}, \eqref{eq. 12}, \eqref{eq. 796}, \eqref{eq. 810} and \eqref{eq. 92146} that
\[\|\phi\|_{\sigma}\leqslant \|\phi\|_{\sigma\otimes_{\C}\tau}=n\cdot\|\phi\|_T.\]

If $\ker\phi$ is of type $F$, then by Theorem \ref{thm. euler characteristic}, 
\[\|\phi\|_{\sigma}=-n\cdot \chi(\ker\phi)=n\cdot \|\phi\|_T.\qedhere\]
\end{proof}

\section{Application to free-by-cyclic and $3$-manifold groups}\label{sec. application}

\subsection{Free-by-cyclic groups}
Let $G$ be a (finitely generated free)-by-(infinite cyclic) group. Then $G$ is locally indicable. In particular, the Thurston norm $\|\cdot\|_T$ of $G$ is well-defined.

\begin{theorem}\label{thm. free-by-cyclic}
For any (finitely generated free)-by-(infinite cyclic) group $G$ and any representation $\sigma\colon  G\rightarrow \GL_n(D)$, the function $\|\cdot\|_{\sigma}$ is a semi-norm.

Moreover, if $D=\C$ is the field of complex numbers, then for every $\phi\in H^1(G,\Z)$,
\begin{equation}\label{eq. 1000}
    \|\phi\|_{\sigma}\leqslant n\cdot \|\phi\|_T
\end{equation}
and equality holds when $\phi$ is a fibered character, i.e., when $\ker\phi$ is finitely generated.
\end{theorem}

\begin{proof}
First, suppose $G_{\mathrm{fab}}=\Z$. Let $\phi\in H^1(G,\Z)$ be the unique (up to sign) primitive integral character. Then $\ker \phi$ is a finitely generated free group, and thus is of type $F$ and satisfies $\chi(\ker\phi)\leqslant 0$. That $\|\cdot\|_{\sigma}$ is a semi-norm follows from Lemma \ref{lem. rkGfb=1 case}. Theorem \ref{thm. inequality for apspherical group} implies that
\[\|\phi\|_{\sigma}=-n\cdot\chi(\ker\phi)=n\cdot\|\phi\|_T.\]

Suppose $\rank G_{\mathrm{fab}}\geqslant 2$. Since $G$ is aspherical (see, e.g., \cite{funke2018alexander}*{Lemma 3.1}), Proposition \ref{prop. semi-norm} implies that $\|\cdot\|_{\sigma}$ is a semi-norm. Since every finitely generated subgroup of $G$ is of type $F$ \cite{feighn1999mapping}, a character $\phi\in H^1(G,\Z)$ is fibered if and only if $\ker\phi$ is of type $F$. Moreover, by \cite{jaikin2020universality}*{Theorem 1.1 and Theorem 3.7 (2)}, $G$ is (Lewin type $F$)-by-(infinite cyclic). Inequality \eqref{eq. 1000}, as well as the equality for fibered characters, follows from Theorem \ref{thm. inequality for apspherical group}.
\end{proof}

\subsection{3-manifold groups}
Let $G$ be the fundamental group of a closed connected orientable $3$-manifold $M$ that fibers over $S^1$. Then $b_1(G)>0$ and thus Lemma \ref{lem. eg of locally indicable} implies that $G$ is locally indicable. In particular, the Thurston norm $\|\cdot\|_T$ of $G$ is well-defined. The goal of this subsection is the following.

\begin{theorem}\label{thm. 3-mnfl}
Suppose that $G$ is the fundamental group of a closed connected orientable $3$-manifold $M$ that fibers over $S^1$. Then $\|\cdot\|_{\sigma}$ is well-defined. Moreover,

\begin{enumerate}[label=(\roman*)]
    \item\label{item. 3-mnfl 1} if $M\neq S^1\times S^2$, then for any representation $\sigma\colon G\rightarrow \GL_n(D)$ of $G$ over a skew field $D$, $\|\cdot\|_{\sigma}$ is a semi-norm;
    \item\label{item. 3-mnfl 2} if $D=\C$, then for every $\phi\in H^1(G,\Z)$, 
\begin{equation}\label{eq. 26}
    \|\phi\|_{\sigma}\leqslant n\cdot\|\phi\|_T
\end{equation}
where$\|\cdot\|_T$ is the Thurston norm of $G$, and where
equality holds if $\phi$ is a fibered character.
\end{enumerate}
\end{theorem}

\begin{proof}
There is a closed surface $S$  such that $M$ decomposes as a fiber bundle $S\rightarrow M\rightarrow S^1$. In particular, $M$ is a mapping torus of a cellular self map of a finite connected CW-complex. That $\|\cdot\|_{\sigma}$ is well-defined thus follows from Proposition \ref{prop. mapping torus and vanishing betti number}.

We have $G=\pi_1(S)\rtimes \Z$. If $\rank G_{\mathrm{fab}}=1$, then items \ref{item. 3-mnfl 1} and \ref{item. 3-mnfl 2} follow from Theorem \ref{thm. euler characteristic} and Lemma \ref{lem. rkGfb=1 case}. Below we assume $\rank G_{\mathrm{fab}}\geqslant 2$.

Let $\widetilde{M}$ be the universal cover of $M$. By the proof of  \cite{mcmullen2002alexander}*{Theorem 5.1}, $\widetilde{M}$ has a $G$-equivariant CW structure whose cellular chain complex has the form

\begin{tikzcd}
C_{\ast}\colon & 0 \arrow[r] & C_3 \arrow[r,"\partial_3"] & C_2 \arrow[r,"\partial_2"] & C_1 \arrow[r,"\partial_1"] & C_0 \arrow[r] & 0,
\end{tikzcd}

\noindent
where there are $\Z G$-bases $\{p\}$ of $C_0$, $\{e_i\}^k_{i=1}$ of $C_1$, $\{f_i\}^k_{i=1}$ of $C_2$, $\{t\}$ of $C_3$, and there is a generating set $\{g_i\}^k_{i=1}$ of $G$ with the following properties:
\begin{enumerate}[label=(\roman*)]
    \item\label{item. chain complex 1} $\partial _1(e_i)=p\cdot (1-g_i),\, \partial _3(t)=\sum^k_{j=1} f_j\cdot (1-g_j)$.
    \item\label{item. chain complex 2} $X=\{q(g_i)\}^m_{i=1}$ is a basis of $G_{\mathrm{fab}}$ for some $m\leqslant k$, and $q(g_i)=0$ for $i>m$, where $q\colon G\rightarrow G_{\mathrm{fab}}$ is the natural surjection of $G$ onto its maximal free abelian quotient $G_{\mathrm{fab}}$.
\end{enumerate}
We denote the matrix representative of $\partial_{\ast}$ under the above bases by $[\partial_{\ast}]$.

Consider the chain complexes

\begin{tikzcd}
A_{\ast}\colon  & 0 \arrow[r] & \Z G \arrow[r,"1-g_1"] & \Z G \arrow[r] & 0 \arrow[r] & 0 \arrow[r] & 0,\\
B_{\ast}\colon  & 0 \arrow[r] & 0 \arrow[r] & \Z G^{k-1} \arrow[r,"U"] & \Z G^k \arrow[r,"\partial _1"] & \Z G \arrow[r] & 0,\\
E_{\ast}\colon  & 0 \arrow[r] & 0 \arrow[r] & \Z G^{k-1} \arrow[r,"W"] & \Z G^{k-1} \arrow[r] & 0 \arrow[r] & 0,\\
F_{\ast}\colon  & 0 \arrow[r] & 0 \arrow[r] & 0 \arrow[r] & \Z G \arrow[r,"1-g_1"] & \Z G \arrow[r] & 0.
\end{tikzcd}

\noindent
Here, $U$ is the matrix obtained from the matrix $[\partial _2]$ by deleting the first column, and $W$ is obtained from $U$ by deleting the first row.

We then have exact sequences of chain complexes
\begin{equation*}
    0\rightarrow A_{\ast} \rightarrow C_{\ast}\rightarrow B_{\ast} \rightarrow 0,
\end{equation*}
\begin{equation*}
    0\rightarrow E_{\ast} \rightarrow B_{\ast} \rightarrow F_{\ast} \rightarrow 0.
\end{equation*} 

Let $\sigma\colon G\rightarrow \GL_n(D)$ be a representation over a skew field $D$. As before we denote $\ore(DG_{\mathrm{fab}})$ by $D(X)$ to emphasize the role played by $X$. Below, the $q(g_1)$-order will be taken with respect to $X$.

Upon tensoring with $(D(X))^n$ via $\sigma\otimes_{\Z} q$, the above exact sequences become
\begin{equation}\label{eq. exact for 3mnfd 1}
    0\rightarrow A_{\ast}\otimes_{\Z G}(D(X))^n \rightarrow C_{\ast}\otimes_{\Z G}(D(X))^n\rightarrow B_{\ast}\otimes_{\Z G}(D(X))^n \rightarrow 0,
\end{equation}
\begin{equation}\label{eq. exact for 3mnfd 2}
    0\rightarrow E_{\ast}\otimes_{\Z G}(D(X))^n \rightarrow B_{\ast}\otimes_{\Z G}(D(X))^n \rightarrow F_{\ast}\otimes_{\Z G}(D(X))^n \rightarrow 0.
\end{equation} 

Consider the unique non-zero differentials in $A_{\ast}\otimes_{\Z G}(D(X))^n$ and $F_{\ast}\otimes_{\Z G}(D(X))^n$, which are given by the left multiplication by the matrix
\[(\sigma\otimes_{\C}\tau)(1-g_1)=\mathrm{Id}+\sigma(-g_1)q(g_1).\]

Note that every entry of $\sigma(-g_1)q(g_1)$ has $q(g_1)$-order at least $1$. So Lemma \ref{lem. matrix reduction} implies that $(\sigma\otimes_{\C}\tau)(1-g_1)$ is invertible, which in turn implies that $A_{\ast}\otimes_{\Z G}(D(X))^n$ and $F_{\ast}\otimes_{\Z G}(D(X))^n$ are exact. Since $G$ is (type $F$)-by-(infinite cyclic), Corollary \ref{cor. fxz} implies that $C_{\ast}\otimes_{\Z G}(D(X))^n$ is exact. It follows that $B_{\ast}\otimes_{\Z G}(D(X))^n$ is also exact, which in turn implies that $E_{\ast}\otimes_{\Z G}(D(X))^n$ is also exact. Moreover, \eqref{eq. exact for 3mnfd 1} and \eqref{eq. exact for 3mnfd 2} are exact sequences of chain complexes.

Let $\rho_A$ (resp. $\rho_B,\rho_C,\rho_E,\rho_F$) be the Reidemeister torsion of $A_{\ast}\otimes_{\Z G} (D(X))^n$ (resp. $B_{\ast}\otimes_{\Z G} (D(X))^n, C_{\ast}\otimes_{\Z G} (D(X))^n,E_{\ast}\otimes_{\Z G} (D(X))^n,F_{\ast}\otimes_{\Z G} (D(X))^n$). Then (see, e.g., \cite{cohen1973course}*{(17.2)})
\[\rho_C=\rho_A \cdot \rho_B = \rho_A \cdot \rho_E \cdot \rho_F = \Det_{D(X)} (\sigma\otimes_{\Z} q)(W) \cdot (\Det_{D(X)} (\mathrm{Id}-\sigma(g_1)q(g_1)))^{-2}.\]

Let $P\colon D(X)\rightarrow \p(G_{\mathrm{fab}})$ be the polytope homomorphism. Since $W$ is a matrix over $\Z G$, Theorem \ref{thm. single polytope} implies that 
\[P(\Det_{D(X)} (\sigma\otimes_{\Z} q)(W))\in P(D[X^{\pm}])\]
is a single polytope. Therefore,
\[P(\rho_C)\in P(D[X^{\pm}])-P(D[q(g_1)^{\pm}]).\]

We have assumed that $\rank G_{\mathrm{fab}}\geqslant 2$. In particular, $q(g_2)\neq 0$. The above argument with $g_2$ in place of $g_1$ yields that 
\[P(\rho_C)\in P(D[X^{\pm}])-P(D[q(g_2)^{\pm}])\]
and thus $P(\rho_C)$ is a single polytope, which means that $\|\cdot\|_{\sigma}$ is a semi-norm.

\smallskip

We proceed to prove item \ref{item. 3-mnfl 2}. Suppose that $D=\C$ is the field of complex numbers. Let $\phi\in H^1(G,\Z)\smallsetminus\{0\}$. Without loss of generality, we may assume that $\phi$ is a primitive integral character. Let $Y$ be a basis of $G_{\mathrm{fab}}$ such that there is $s\in Y$ with $\phi(s)=1$ and $\phi(y)=0$ for all $y\in Y\smallsetminus\{s\}$. For any right $\Z G$-module $N$, $N\otimes_{\Z G}(D(X))^n$ and $N\otimes_{\Z G}(D(Y))^n$ are naturally isomorphic. Thus, we obtain the following exact sequences from \eqref{eq. exact for 3mnfd 1} and \eqref{eq. exact for 3mnfd 2}
\[  0\rightarrow A_{\ast}\otimes_{\Z G}(\C(Y))^n \rightarrow C_{\ast}\otimes_{\Z G}(\C(Y))^n\rightarrow B_{\ast}\otimes_{\Z G}(\C(Y))^n \rightarrow 0,\]
\[  0\rightarrow E_{\ast}\otimes_{\Z G}(\C(Y))^n \rightarrow B_{\ast}\otimes_{\Z G}(\C(Y))^n \rightarrow F_{\ast}\otimes_{\Z G}(\C(Y))^n \rightarrow 0,\]
which imply that
\begin{equation}\label{eq. 21}
    \|\phi\|_{\sigma}=\deg_s\Det_{\C(Y)} (\sigma\otimes_{\Z} q)(W) - 2\cdot \deg_s\Det_{\C(Y)} (\mathrm{Id}-\sigma(g_1)q(g_1)).
\end{equation}

Let $\tau:\C G\rightarrow \D_G$ be the natural embedding of $\C G$ into its Linnell skew field. The same argument with $\tau$ in place of $\sigma$ yields
\begin{equation}\label{eq. 22}
    \|\phi\|_T=\deg_s\Det_{\D_G(Y)} (\tau\otimes_{\Z} q)(W) - 2\cdot \deg_s\Det_{\D_G(Y)} (\mathrm{Id}-\tau(g_1)q(g_1)).
\end{equation}
We prove that
\begin{equation}\label{eq. the minus part}
    \deg_s\Det_{\C(Y)} (\mathrm{Id}-\sigma(g_1)q(g_1))=n\cdot\deg_s\Det_{\D_G(Y)} (\mathrm{Id}-\tau(g_1)q(g_1)).
\end{equation}

Write $q(g_1)$ as a monomial in $Y$:
\[q(g_1)=s^r\cdot\prod_{y\in Y\smallsetminus\{s\}}y^{r_y},\]
where we use multiplicative notation for the abelian group $G_{\mathrm{fab}}$. 

First consider the case $r=0$. Then $\deg_s\Det_{\C(Y)} (\mathrm{Id}-\sigma(g_1)q(g_1))$ is either $0$ or $-\infty$. We have shown that $\mathrm{Id}-\sigma(g_1)q(g_1)$ is invertible, so 
\[\deg_s\Det_{\C(Y)} (\mathrm{Id}-\sigma(g_1)q(g_1))=0.\]
Clearly, $\deg_s\Det_{\D_G(Y)} (\mathrm{Id}-\tau(g_1)q(g_1))=0$. Thus, \eqref{eq. the minus part} holds in this case.

Consider the case $r\neq 0$. Without loss of generality we may assume $r>0$ (the case $r<0$ can be analyzed in the same way). Think of $\Det_{\C(Y)} (\mathrm{Id}-\sigma(g_1)q(g_1))$ as a polynomial in $s$ with coefficients in $\C(Y\smallsetminus\{s\})$. Then the highest power of $s$ in $\Det_{\C(Y)} (\mathrm{Id}-\sigma(g_1)q(g_1))$ is $s^{nr}$ with coefficient $\Det_{\C(Y)}\sigma(g_1)\cdot \prod_{y\in Y\smallsetminus\{s\}}y^{nr_y}$. The lowest power of $s$ in $\Det_{\C(Y)} (\mathrm{Id}-\sigma(g_1)q(g_1))$ is $s^0=1$ with coefficient $1$. Thus,
$\deg_s\Det_{\C(Y)} (\mathrm{Id}-\sigma(g_1)q(g_1))=nr$. Since $\deg_s\Det_{\D_G(Y)} (\mathrm{Id}-\tau(g_1)q(g_1))=r$, equation \eqref{eq. the minus part} also holds in this case.

Theorem \ref{thm. extension} extends the representation $\sigma\otimes_{\C}\tau\colon G\rightarrow \mathrm{GL}_n(\D_G)$ to a ring homomorphism $\ts\colon \D_G\rightarrow M_n(\D_G)$. By repeatedly using Corollary \ref{cor. ring homo extension}, we further extend $\ts$ to a ring homomorphism (still denoted by) $\ts\colon \D_G(Y)\rightarrow M_n(\D_G(Y))$ such that $\ts(y)=y\cdot \mathrm{Id}$ for all $y\in Y$. Then
\[\ts((\tau\otimes_{\Z}q)(W))=(\sigma\otimes_{\C}\tau\otimes_{\Z}q)(W).\]

Note that $\pi_1(S)$ is Lewin. Indeed, there is a homomorphism $\pi_1(S)\rightarrow \Z$ whose kernel is the fundamental group of a non-compact surface, and thus is free. Therefore, $\pi_1(S)$ is a semi-direct product of a free group with $\Z$. It follows that $\pi_1(S)$ is Lewin and thus so is $G$ \cite{jaikin2020universality}*{Theorems 1.1 and 3.7}.

Lemmata \ref{lem. key lem} and \ref{lem. extension and degree} then imply
\begin{align*}
    \deg_s\Det_{\C(Y)} (\sigma\otimes_{\Z} q)(W)\leqslant&\deg_s\Det_{\D_G(Y)} (\sigma\otimes_{\C}\tau\otimes_{\Z} q)(W)\\
    =&n\cdot \deg_s\Det_{\D_G(Y)} (\tau\otimes_{\Z} q)(W),
\end{align*}
which, together with \eqref{eq. 21}, \eqref{eq. 22}, \eqref{eq. the minus part} finishes the proof of \eqref{eq. 26}.

If $\phi$ is a fibered character, then $\ker\phi$ is the fundamental group of some closed surface \cite{stalling1962fibering} and thus is of type $F$. Theorem \ref{thm. euler characteristic} then implies that \[\|\phi\|_{\sigma}=-n\cdot\chi(\ker\phi)=n\cdot\|\phi\|_T.\qedhere\]
\end{proof}

\bibliographystyle{alpha}
\bibliography{bin_refs}

\end{document}